\documentclass[12pt]{amsart}

\title{A Generalization of Seifert Geometry Based on the Siegel Upper Half-Space}

\author{Qing Lan}
\address{Beijing International Center for Mathematical Research}
\email{lanqing@stu.pku.edu.cn}

\date{\today}

\usepackage{amsmath, amsfonts, amssymb, amsthm, amscd, graphicx, float, epstopdf, mathrsfs
,esint
}

\usepackage[pdfpagemode={UseOutlines},bookmarks=true,bookmarksopen=true, bookmarksopenlevel=0,bookmarksnumbered=true,hypertexnames=false, colorlinks,linkcolor={blue},citecolor={blue},urlcolor={blue}, pdfstartview={FitV},unicode,breaklinks=true,backref=page]{hyperref}

\usepackage{tikz-cd}

\usepackage{fullpage}
\usepackage[all]{xy}
\usepackage{graphicx}
    \newcommand{\BC}{{\mathbb {C}}} 
    \newcommand{\BE}{{\mathbb {E}}} 
     \newcommand{\BH}{{\mathbb {H}}}

     \newcommand{\BN}{{\mathbb {N}}}
     
    \newcommand{\BQ}{{\mathbb {Q}}} \newcommand{\BR}{{\mathbb {R}}}
    \newcommand{\BS}{{\mathbb {S}}}

     \newcommand{\BZ}{{\mathbb {Z}}}

    \newcommand{\CG}{{\mathcal {G}}} \newcommand{\CH}{{\mathcal {H}}}

    \newcommand{\CM}{{\mathcal {M}}} 
    \newcommand{\CO}{{\mathcal {O}}} 
     \newcommand{\CR}{{\mathcal {R}}}

     \newcommand{\CX}{{\mathcal {X}}}
     \newcommand{\CZ}{{\mathcal {Z}}}

     \newcommand{\fH}{{\mathfrak{H}}}

    \newcommand{\id}{{\mathrm{id}}}

    \newcommand{\Isom}{{\mathrm{Isom}}}

    \newcommand{\Stab}{{\mathrm{Stab}}}

    \newcommand{\vol}{{\mathrm{vol}}}

    \theoremstyle{plain}
    \newtheorem{thm}{Theorem}[section] \newtheorem{cor}[thm]{Corollary}
    \newtheorem{lem}[thm]{Lemma}  \newtheorem{prop}[thm]{Proposition}
     \newtheorem{defn}[thm]{Definition}

\newtheorem{conv}[thm]{Convention}

    \newtheorem{lem-defn}[thm]{Lemma-Definition}

\theoremstyle{remark} \newtheorem{remark}[thm]{Remark}
\theoremstyle{remark} 
\theoremstyle{remark} 

    \newcommand{\Poincare}{Poincar\'{e}~}
    
    \numberwithin{equation}{section}

\newcommand\sltrt{\widetilde{\mathrm{SL}(2,\mathbb{R})}}
\newcommand\sltr{\mathrm{SL}(2,\mathbb{R})}
\newcommand\sltrtxzr{\widetilde{\mathrm{SL}(2,\mathbb{R})} \times_{\mathbb{Z}}\mathbb{R}}

\newcommand\pspfr{\mathrm{PSp}(4,\mathbb{R})}
\newcommand\spfr{\mathrm{Sp}(4,\mathbb{R})}
\newcommand\spfrt{\widetilde{\mathrm{Sp}(4,\mathbb{R})}}
\newcommand\utwo{\mathrm{U}(2)}
\newcommand\sutwo{\mathrm{SU}(2)}
\newcommand\utwot{\widetilde{\mathrm{U}(2)}}

\newcommand\Psptnr{\mathrm{PSp}(2n,\mathbb{R})}
\newcommand\Sptnr{\mathrm{Sp}(2n,\mathbb{R})}
\newcommand\Sptnrt{\widetilde{\mathrm{Sp}(2n,\mathbb{R})}}

\newcommand\Un{\mathrm{U}(n)}
\newcommand\Sun{\mathrm{SU}(n)}
\newcommand\Unt{\widetilde{\mathrm{U}(n)}}

\newcommand\gxzghrh{\CG\times_{Z(\CG)}\CH\CR/\CH}
\newcommand\hrh{\CH\CR/\CH}

\newcommand\dett{\widetilde{\mathrm{det}}}

\begin{document}
\maketitle
\begin{abstract}
The Seifert geometry, $\widetilde{\mathrm{SL}(2,\mathbb{R})}$-geometry, is one of Thurston's eight $3$-dimensional geometries. It fibers over the hyperbolic plane $\mathbb{H}^2$, which is a special case of the Siegel upper half-space $\mathrm{Sp}(2n,\mathbb{R})\curvearrowright {\mathfrak{H}}_n$. In this paper we construct an analogous geometry fibering over the Siegel upper half-space, and provide a volume formula for some manifolds with this geometry. For $n=2$, a prototype is constructed via the normal bundle of an equivariant embedding into a Grassmannian manifold. It turns out that this geometry is the homogeneous space given by a central extension of $\widetilde{\mathrm{Sp}(2n,\mathbb{R})}$, modulo its maximal compact subgroup. The volume of a Siegel--Seifert closed manifold of this geometry is shown to be the length of the fiber circle times the Euler characteristic of the base manifold, up to a sign. Examples of Siegel--Seifert manifolds are provided, and it is shown that the volume of representations for this geometry is constant on every path-connected component of the representation space.


\end{abstract}

%Parallel to $\widetilde{\mathrm{SL}(2,\mathbb{R})}$-geometry fibering over the hyperbolic plane, we construct a  geometry fibering over the Siegel upper half-space $\mathrm{Sp}(2n,\mathbb{R})\curvearrowright {\mathfrak{H}}_n$, and provide a volume formula for some manifolds with this geometry. For $n=2$, a prototype is constructed via the normal bundle of an equivariant embedding into a Grassmannian manifold. It turns out that this geometry is the homogeneous space given by a central extension of $\widetilde{\mathrm{Sp}(2n,\mathbb{R})}$, modulo its maximal compact subgroup. After fixing a convention for the invariant measure, the volume of a ``Siegel--Seifert'' closed manifold of this geometry is given by the length of the fiber circle times the Euler characteristic of the base manifold, up to a sign. Examples of Siegel--Seifert manifolds are provided, and rigidity of volume of representations is discussed. 

%here

\section{Introduction}

\subsection{Background. } 

A \textbf{model geometry} $(G,X)$, in the sense of Thurston \cite[Definition~3.8.1]{thurston}, is a Lie group $G$ acting by diffeomorphisms on a manifold $X$, such that 

(i) $X$ is connected and simply connected; 

(ii) $G$ acts transitively (and effectively) on $X$  with compact point stabilizers;

(iii) $G$ is not contained in any larger group of diffeomorphisms of $X$
with compact stabilizers of points, i.e. $G$ is maximal; 

(iv) there exists at least one compact manifold modeled on $(G,X)$.

Thurston proved that there are $8$ maximal geometries in dimension $3$, namely 
the Euclidean geometry $\BE^3$, spherical geometry $\BS^3$, hyperbolic geometry $\BH^3$, Seifert geometry $\widetilde{\mathrm{SL}(2,\mathbb{R})}$, product geometries $\BH^2\times \BE^1$ and $\BS^2\times \BE^1$, nilgeometry $\BN\mathrm{il}$ and solgeometry $\BS\mathrm{ol}$. With the exception of $\BH^3$ and $\BS\mathrm{ol}$, the remaining $6$ geometries constitute the Seifert fiber spaces (or Seifert fibrations), fibering over $2$-dimensional orbifolds with fiber $S^1$. For an overview of $3$-dimensional geometries, see \cite{scott} or \cite{martelli}.

Four dimensional geometries are classified by Filipkiewicz \cite{filipkiewicz}: there are 
%$18$ geometries together with one infinite family of closely related geometries. 
$18$ distinct geometries along with one infinite family of closely related geometries. 
Unlike $3$-dimensional case in which the Geometrization Conjecture holds, in dimension $4$ manifolds with geometric decomposition are  relatively  rare \cite[Theorem~7.2]{hillman}. The notion of Seifert fibration has been generalized to dimension $4$: a $4$-manifold is Seifert fibered if it is the total space of an orbifold bundle with general fiber a torus or Klein bottle over a $2$-orbifold \cite[Section~7.7]{hillman}. Possible types of geometries that a closed orientable Seifert $4$-manifold can admit are classified by Ue in \cite{ue90,ue91}, in terms of types of the base orbifold.  Five dimensional geometries are classified by Geng in \cite{5dimgeo}.

In this paper we consider a generalization of the Seifert geometry, fibering over the Siegel upper half-space $\mathrm{Sp}(2n,\mathbb{R})\curvearrowright {\mathfrak{H}}_n$ \cite{siegelsympgeo}. The Siegel upper half-space $\mathrm{Sp}(2n,\BR)\curvearrowright \fH_n$ is a generalization of the hyperbolic plane $\mathrm{SL}(2,\mathbb{R}) \curvearrowright \BH^2$. 
As a homogeneous space, the Siegel upper half-space is just $\mathrm{Sp}(2n,\BR)$ modulo its maximal compact subgroup, which is an embedded $\mathrm{U}(n)$. 

The Lie group 
%$\CG\times_{Z(\CG)}\CH\CR/\CH$
 we construct is a \textbf{topological group extension} of $\Psptnr$ by $\BR$, in the sense of \cite[Section~5]{hu} (See Propositions \ref{prop:exactseq} and \ref{prop:topgrpext}). It is proved in \cite{hu} (see also the survey \cite{stasheff}) that the second continuous cohomology $H^2_c(G;A)$ of a topological group $G$ with coefficients in a continuous $G$-module $A$, classifies topologically split group extensions, i.e. topological group extensions such that, considered as a principal bundle, the bundle is trivial. 
In our case, for the Lie group $\BR$, its classifying space is just a point,  so the sequence in Proposition \ref{prop:exactseq} splits topologically. 

For   a connected, non-compact, simple Lie group $G$ with finite center, $H^2_c(G;\BR)\neq 0$ if and only if $G$ is of Hermitian type (\cite[Lemme~1]{seconddegcontcoh}, see also \cite{thirddegcontcoh});  
if this is the case, $H^2_c(G;\BR)$ is one dimensional. It follows that the only possibilities  admitting topologically split group extensions by $\BR$ (where the action by $G$ is trivial), are the $6$ families listed after \cite[Lemme~1]{seconddegcontcoh}. 
As a first step, we are performing extensions to one of these families and considering the corresponding geometry.

\subsection{Main results. }

We construct a  central extension by $ \BR$ of $\CG:=\Sptnrt$, denoted by  $\gxzghrh$,  
\footnote{The group $\gxzghrh$ is defined in Lemma-Definition \ref{lemdef:ghr} and Definition \ref{def:gxzghrh}.  Here $\CG:=\Sptnrt$, and  $\CH:=\Sun$ is a specified  maximal compact subgroup of $\CG$. The full definition is too complicated to state here. }
 acting on the model space  $\CX:=\Sptnrt/\Sun$, realizing  $\CX$ as $\gxzghrh$ modulo its maximal compact subgroup. The model space is the total space of a principal $\BR$-bundle $\nu: \CX \to  \fH_n  $ over the Siegel upper half space, and the action of $\CG\times_{Z(\CG)}\CH\CR/\CH$ preserves this fibering structure. Based on this fibering structure, we prove a volume formula for a reasonably general class  of closed manifolds called Siegel--Seifert manifolds (Definition \ref{def:siegelseifert}), after fixing 
%such 
an invariant volume form (Conventions \ref{conv:volr} and \ref{conv:voleu}):

\begin{thm}%[Proved in Subsection \ref{subsec:5.3}]
For any Siegel--Seifert subgroup $\Gamma \subset  \CG\times_{Z(\CG)}\CH\CR/\CH$, the volume of $\Gamma\backslash \CX$ is given by 
\[
\vol(\Gamma\backslash \CX)=\pm\vol((\Gamma\cap\ker\eta)\backslash  
%(\CH\CR/\CH)
\BR    )\cdot \chi(\eta(\Gamma)). 
\]

In particular, for any Siegel--Seifert subgroup $\Gamma$ arising from (Definition \ref{def:arisesfrom}) $\Psptnr$, the volume of $\Gamma\backslash \CX$ is given by 
\[
\vol(\Gamma\backslash \CX)=\pm \chi(\eta(\Gamma)). 
\]

Here  $\eta:  \CG\times_{Z(\CG)}\CH\CR/\CH \to \CG/Z(\CG)\cong \Psptnr$ is given in Proposition \ref{prop:exactseq}. 
\label{thm:vol}
\end{thm}

Examples of Siegel--Seifert subgroups are constructed in Section \ref{sec:6} by showing that 
for any cocompact torsion-free lattice $L$ in $\Psptnr$, 
the group $\Gamma=L^\uparrow$ arising from $L$  is Siegel--Seifert (Proposition \ref{prop:siegelseifertarisingfrom}).

Since  $\eta(\Gamma)\backslash\fH_n$ is an Eilenberg--Maclane space, we have
$\chi(\eta(\Gamma)\backslash\fH_n)=\chi(\eta(\Gamma))$.  
Here we are using the Euler characteristic of torsion-free groups; for the definition of Euler characteristic (with values in $\BQ$) of a group with torsion, see Chapter IX in \cite{gtm87}. 
It follows that for Siegel--Seifert subgroups $\Gamma_i$ arising from $\Psptnr$, the volumes of $\Gamma_i\backslash \CX$ are rationally related.

In addition, we show the following    result with the help of \cite{dlsw}. This is Corollary  \ref{cor:oddevenrigid}, proved in Section \ref{sec:7}. 

\begin{thm}[Corollary of \cite{dlsw}]
  For any closed oriented smooth manifold $M$ (of applicable dimension), the volume of representations
$$\operatorname{vol}_{\CG\times_{Z(\CG)}\CH\CR/\CH}:\mathrm{Rep}(\pi_1(M),\CG\times_{Z(\CG)}\CH\CR/\CH)\to\mathbb{R}$$
is constant on every path-connected component of the representation space $\mathrm{Rep}(\pi_1(M),\CG\times_{Z(\CG)}\CH\CR/\CH)$.
\label{thm:rigid}
\end{thm}

Theorem \ref{thm:vol} generalizes the classical volume formula for $\sltrt$-geometry  \cite[Corollary~4.7.3 and its proof]{thurston}: For any discrete group $\Gamma$ of isometries of  $X=\sltrt$, if $p(\Gamma)$ is a discrete subgroup of $\mathrm{Isom}(\BH^2)$, we have 
\[
\mathrm{vol}(X/\Gamma)=\mathrm{area}(\mathbb{H}^{2}/p(\Gamma))\operatorname{length}(\mathbb{R}/(\Gamma\cap\ker p)). 
\]
Here $p$ is given in the following exact sequence \cite{scott}: 
\[
0\to\mathbb{R}\to\mathrm{Isom}\left(\sltrt\right)\xrightarrow{p}\mathrm{Isom}(\BH^2)\to 0. 
\]

This exact sequence for  $\mathrm{Isom}\left(\sltrt\right)$ and a similar exact sequence for $\widetilde{\mathrm{SL}(2,\mathbb{R})}\times_\BZ \BR$ are generalized to our case in the following proposition. 

Indeed, there is a group isomorphic to $ \mathbb{R}$, which embeds in $ \CG\times_{Z(\CG)}\CH\CR/\CH$ as a central subgroup, 
% via the map  $r\CH\mapsto (\mathrm{id}, r\CH)$, 
making the following exact sequence hold: 
%We denote this subgroup by 
%$
%\theta(\{\mathrm{id}\}\times \CH\CR/\CH)
%$. 

\begin{prop}%[Proved in Subsection \ref{subsec:4.2}]
There is an exact sequence of groups
\[
0\to 
%\theta(\{\mathrm{id}\}\times \CH\CR/\CH)
\BR \to \CG\times_{Z(\CG)}\CH\CR/\CH \stackrel{\eta}{\to} \CG/Z(\CG) \to 0, 
\]
where 
%$\theta(\{\mathrm{id}\}\times \CH\CR/\CH)\cong \mathbb{R}$,  and 
$\CG/Z(\CG)\cong \Psptnr$. 
\label{prop:exactseq}
\end{prop}

For any subgroup $\Gamma \subset  \CG\times_{Z(\CG)}\CH\CR/\CH$, there is an exact sequence of groups 
\[
0\to \Gamma\cap\ker \eta   \to \Gamma   \to \eta(\Gamma)\to 0. 
\]

Similarly, for   a discrete group $\Gamma'$ of isometries of $\sltrt$ acting freely and
with quotient $M$, the foliation of $\sltrt$ descends to a foliation of $M$. There is also an induced exact sequence
\[
0\to \Gamma'\cap\mathbb{R} \to \Gamma'\to p(\Gamma')\to 0. 
\]
As $\Gamma'$ is discrete, $\Gamma'\cap\mathbb{R}$ is infinite cyclic or trivial. If $\Gamma'\cap\mathbb{R}$ is infinite cyclic, we have that the foliation of $M$ is a Seifert fibration, and in particular a foliation into circles, fibering over a $2$-dimensional orbifold. This motivates the following definition of Siegel--Seifert subgroups, the groups in consideration:

\begin{defn}
A subgroup $\Gamma \subset  \CG\times_{Z(\CG)}\CH\CR/\CH$ is \textbf{Siegel--Seifert} if

(i) $\Gamma \subset  \CG\times_{Z(\CG)}\CH\CR/\CH$ is discrete and torsion-free, and the action $\Gamma\curvearrowright 
\Sptnrt/\Sun=\CX$ is cocompact and properly discontinuous, 

(ii) $\eta(\Gamma) \subset \CG/Z(\CG) $ is discrete and torsion-free, and the action $\eta(\Gamma)\curvearrowright \Sptnrt/\Unt =\fH_n$ is cocompact and properly discontinuous, and

(iii) $\Gamma\cap\ker\eta$ is infinite cyclic.

The corresponding manifold $\Gamma\backslash \mathcal{X}$ is called a \textbf{Siegel--Seifert} manifold. 
Here  $\eta:  \CG\times_{Z(\CG)}\CH\CR/\CH \to \CG/Z(\CG)\cong \Psptnr$ is given in Proposition \ref{prop:exactseq}. 
\label{def:siegelseifert}
\end{defn}

For a Siegel--Seifert manifold $\Gamma\backslash \mathcal{X}$, as $\Gamma\cap\ker\eta$ is infinite cyclic, the foliation on $\mathcal{X}$ into lines descends to a foliation into circles on $\Gamma\backslash \mathcal{X}$, fibering over the base manifold $\eta(\Gamma)\backslash\fH_n$. 

We've imposed strong requirements on the definition. In practice, however, subgroups that are not torsion-free or subgroups that are just cofinite can also be considered, as in the case of hyperbolic orbifolds and Seifert fibrations fibering over orbifolds.

It is important to acknowledge that there are numerous other generalizations of Seifert fibration that have been proposed in the literature. 
%While some of these share similar terminologies, they often diverge significantly in their construction and conceptual foundations. 
One of the most widely studied generalizations is the one mentioned in the first subsection, where a $4$-manifold is considered Seifert fibered if it is the total space of an orbifold bundle with general fiber a torus or Klein bottle over a $2$-orbifold \cite[Section~7.7]{hillman}. 
Other generalizations  include a generalization to Alexandrov $3$-spaces, which was first proposed in \cite{alexandrov1} and classified in \cite{alexandrov2}. Furthermore, originating from the study of positively curved manifolds, in \cite{taimanov} a  generalized Seifert fibration is defined to be a ``bundle with singular locus'', with   Eshchenburg $7$-manifolds \cite{eschenburg} (and also any principal $S^1$-bundle over a manifold) serving as a special case. 
Although these various definitions share a common terminology, they are conceptually distinct from the definition we have introduced, and the focus of this paper will remain on the latter. 
% as we move toward  the overall structure of the proof outlined in the next subsection.

\subsection{A roadmap for the proof. }

In this paper we first discuss an explicit construction of such a  model based on $\spfr\curvearrowright \fH_2$, and then for general $n$. It turns out that this model is the homogeneous space given by a central extension 
%$\gxzghrh$
 of $\widetilde{\mathrm{Sp}(4,\mathbb{R})}$, modulo its maximal compact subgroup. 
% More generally, the central extension is done for $\CG:=\Sptnrt$, and such a model for general $n$ is constructed. 
We use an   equivariant embedding from $\fH_2$ to a Grassmannian manifold $P_2$ taken from p.94 of \cite{hulek}. 
\textit{The} normal bundle, as a subbundle of $TP_2|_{\fH_2}$ of real rank $2$, is specified by two global sections and is proved to be invariant under the tangent map of $\spfr\curvearrowright P_2$. 
We then obtain an action 
\[
\spfr\curvearrowright \fH_2\times S^1, 
\]
which can be lifted to the universal covering spaces
\[
\spfrt\curvearrowright \fH_2\times \BR. 
\]
It is proved in Theorem \ref{thm:models} that the actions $\spfr\curvearrowright \fH_2\times S^1$ and $\spfrt\curvearrowright \fH_2\times \BR$ are both transitive and have compact stabilizer isomorphic to $\sutwo$.

Inspired by Theorem \ref{thm:models},    in Subsection \ref{subsec:3.4}, 
we perform a central extension by $\CH\CR/\CH\cong \BR$ on $\CG:=\Sptnrt$ to obtain a group $\gxzghrh$ (defined in Lemma-Definition \ref{lemdef:ghr} and Definition \ref{def:gxzghrh}) 
acting on the model space  $\CX=\Sptnrt/\Sun$ (Definition \ref{def:extendedact})
 realizing  $\CX$ as $\gxzghrh$ modulo its maximal compact subgroup (Lemma \ref{lem:maximalcompactmodel}). The following Proposition \ref{prop:extendedact}, summarizing this model,  is proved in Subsection \ref{subsec:3.4}:

\begin{prop}%[Proved in Subsection \ref{subsec:3.4}]
The extended action 
\[
 \CG\times_{Z(\CG)}\CH\CR/\CH   \to    \mathrm{Diff}(\CX)
%=\mathrm{Diff}(\Sptnrt/\Sun) 
\]
 is transitive, and has compact stabilizer isomorphic to 
${\Un}/{\{\pm \id\}}$. 
\label{prop:extendedact}
\end{prop}

Proposition \ref{prop:exactseq} is proved in Subsection \ref{subsec:4.2}. We observe that 
the action of $\CG\times_{Z(\CG)}\CH\CR/\CH$ preserves the fibering structure of 
the trivial principal $\BR$-bundle $\nu: \CX \to  \fH_n  $, and 
the induced action of $\CG\times_{Z(\CG)}\CH\CR/\CH$ on the base space coincides with the usual action of $\Sptnr$ on $\Sptnr/\Un$
(Lemma \ref{lem:descendact}). 
Setting 
\[
\theta: \CG\times \CH\CR/\CH\to \CG\times_{Z(\CG)}\CH\CR/\CH
\]
to be the quotient map, 
the group $\CH\CR/\CH\cong \mathbb{R}$ embeds in $ \CG\times_{Z(\CG)}\CH\CR/\CH$ as a central subgroup via the map  $r\CH\mapsto (\mathrm{id}, r\CH)$, making the  exact sequence in Proposition \ref{prop:exactseq} hold.

The Theorem \ref{thm:vol} relies on 
Goetz's result \cite{goetz} summarized in Subsection \ref{subsec:5.2}, which  shows that the invariant measure in our model is given by a ``product measure'' (also defined in Subsection \ref{subsec:5.2}) of the base and the fiber. After fixing such an invariant volume form (Conventions \ref{conv:volr} and \ref{conv:voleu}), using the Chern--Gauss--Bonnet theorem \cite{cherngaussbonnet}, the proof of Theorem \ref{thm:vol} is finished in Subsection \ref{subsec:5.3}.

In Section \ref{sec:6}, we provide examples of Siegel--Seifert subgroups by showing that 
for any cocompact torsion-free lattice $L$ in $\Psptnr$, 
the group $\Gamma=L^\uparrow$ arising from $L$  is Siegel--Seifert (Proposition \ref{prop:siegelseifertarisingfrom}). 
Note that starting with any cocompact lattice in $\Psptnr$, by Selberg's Lemma (see \cite[Section~4.8]{morris}), we may replace it with a torsion-free subgroup of finite index which is again a cocompact lattice. 
Classification of (cocompact) lattices for some simple real Lie groups is summarized by  Audibert in \cite{splattices} (see also his thesis). This is briefly sketched in Subsection \ref{subsec:6.2}.

In Section \ref{sec:7}, we establish structure of representation spaces and also  results on the volume of representations associated to $\gxzghrh$, as corollaries to \cite{dlsw}.   The main results of this section are Corollaries \ref{cor:oddeveneuler} and \ref{cor:oddevenrigid}, and the proof is divided to two cases:
For $n$ odd, after fixing identifications $\CH\CR/\CH\to \BR$ and $Z_{\BR}:=Z(\CG)\otimes_{\BZ}\BR\to \BR$, our extended group $\gxzghrh$ and their extended group $\CG_\BR:=\CG\times_{Z(\CG)} Z_\BR$ are both identified with the same quotient of $\CG\times \BR$, and thus Proposition 5.1 and Theorem 7.1 in \cite{dlsw}  in this case directly apply. For $n$ even,  by considering $G=\Sptnrt/\{\pm\id\}$ instead, our construction and their construction coincide, and similar results follow. 
As a consequence of considerations on representation volume, we show in Proposition \ref{prop:seifertpropertyd} that any Siegel--Seifert manifold $\Gamma\backslash \mathcal{X}$ has Property D (Definition \ref{defn:propertyd}), in Subsection \ref{subsec:7.5}.

\subsection{Organization. }

%This paper is organized as follows: 
Section \ref{sec:2} is an introduction to Seifert geometry with emphasis on the similarities between the Seifert geometry and our model. The construction of our model is done in Section \ref{sec:3}:  Subsections \ref{subsec:3.1}, \ref{subsec:3.2} and \ref{subsec:3.3} define the preliminary model $\spfrt\curvearrowright \fH_2\times \BR$ and prove the identification $\spfrt\curvearrowright \fH_2\times \BR\cong \spfrt\curvearrowright \spfrt/\sutwo$. Starting from Subsection \ref{subsec:3.4} we turn to the case for general $n$. Subsection \ref{subsec:3.4} performs the central extension obtaining our model  $\CG\times_{Z(\CG)}\CH\CR/\CH   \to    \mathrm{Diff}(\CX)=\mathrm{Diff}(\Sptnrt/\Sun) $. 

Section \ref{sec:4} is devoted to the fibering structure and the induced exact sequence. As a consequence of the exact sequence, Subsection  \ref{subsec:4.3} discusses $\CG\times_{Z(\CG)}\CH\CR/\CH$ as a topological group extension.  The considered Siegel--Seifert subgroups are defined in Subsection \ref{subsec:5.1}. Work by Goetz on the product measure is summarized in Subsection \ref{subsec:5.2}, which is an essential tool in the proof of the volume formula in Subsection \ref{subsec:5.3}.

In Section \ref{sec:6}, we provide examples of Siegel--Seifert subgroups by showing that 
for any cocompact torsion-free lattice $L$ in $\Psptnr$, 
the group $\Gamma=L^\uparrow$ arising from $L$  is Siegel--Seifert (Proposition \ref{prop:siegelseifertarisingfrom}). This is discussed in Subsection \ref{subsec:6.1}. 
Classification of (uniform) lattices in $\Sptnr$ up to (wide) commensurability,  summarized by  Audibert in \cite{splattices},  is briefly sketched in Subsection \ref{subsec:6.2}.

In Section \ref{sec:7},  we exhibit structure of representation spaces and also  results on the volume of representations associated to $\gxzghrh$, as corollaries to \cite{dlsw}. 
The main results of this section are Corollaries \ref{cor:oddeveneuler} and \ref{cor:oddevenrigid}. 
%This is done by showing that our extended group $\gxzghrh$ is isomorphic to theirs, for certain choices of $G$ in their paper.  
We begin with a brief introduction to volume of representations. The two cases,  where $n$ is odd or even,  are dealt with separately in Subsections \ref{subsec:7.3} and \ref{subsec:7.4}. 
In Subsection \ref{subsec:7.5}, we show in Proposition \ref{prop:seifertpropertyd} that any Siegel--Seifert manifold $\Gamma\backslash \mathcal{X}$ has Property D (Definition \ref{defn:propertyd}).

\textbf{Acknowledgments. }The author thanks Prof. Yi Liu for proposing the problem and providing suggestions throughout the research, Prof. Wen-Wei Li for his inspiring slides on Euler--\Poincare measure, and Yaoping Xie for suggestions on the product measure. The author thanks the anonymous referee for helpful suggestions.

\textbf{Funding. }This research did not receive any specific grant from funding agencies in the public, commercial, or not-for-profit sectors.

\section{Preliminaries}
\label{sec:2}

The Seifert geometry is one of Thurston's eight geometries, and is an $\BR$-bundle over the hyperbolic plane $\BH^2$. The identity component of its  isometry group is $\widetilde{\mathrm{SL}(2,\mathbb{R})}\times_\BZ \BR$, acting on the  space $\widetilde{\mathrm{SL}(2,\mathbb{R})}$ as follows: the subgroup $\widetilde{\mathrm{SL}(2,\mathbb{R})}$ acts as left multiplication, while the $\BR$ part acts as multiplication on the right by a subgroup of $\widetilde{\mathrm{SL}(2,\mathbb{R})}$ isomorphic to $\BR$. 

We briefly review the definition of Siegel upper half-space, the construction of the Seifert geometry, and results about Seifert geometry motivating our constructions and theorems. 
This section is based on \cite{thurston},  \cite{scott},  \cite{martelli} and  \cite{volofseirep}. 

\subsection{The Siegel upper half-space $\fH_n$ and the stabilizer of $iI_n\in \fH_n$}

The Siegel upper half-space \cite{siegelsympgeo} $\mathrm{Sp}(2n,\BR)\curvearrowright \fH_n$  is a generalization of the hyperbolic plane $\mathrm{SL}(2,\mathbb{R}) \curvearrowright \BH^2$. 
As a homogeneous space, the Siegel upper half-space is just $\mathrm{Sp}(2n,\BR)$ modulo its maximal compact subgroup, which is an embedded $\mathrm{U}(n)$. A concrete construction goes as follows: Let $I_n$ be the $n\times n$ identity matrix. 
The symplectic group $\mathrm{Sp}(2n,\BR)$ is defined to be 
$$\mathrm{Sp}(2n,\BR)=\{M\in M_{2n\times2n}(\BR) | M^{\mathrm{T}}\Omega M=\Omega\},\text{ where } \Omega:= \begin{pmatrix}
0 & I_n \\
-I_n & 0 
\end{pmatrix}.   $$
The \textbf{Siegel upper half-space} $\fH_n$ is defined to be the set of $n\times n$ symmetric matrices over $\BC$ such that the imaginary part is positive definite. Writing $Z\in \fH_n$ as $Z=X+iY$ for real matrices $X,Y$, this means that $Y$ is positive definite. 
The action $\mathrm{Sp}(2n,\BR)\curvearrowright \fH_n$ is given as follows: for $M= \begin{pmatrix}
A & B \\
C & D 
\end{pmatrix} \in  \mathrm{Sp}(2n,\BR)$ decomposed into $n\times n$ matrices $A,B,C,D$, the action is 
\[
\fH_n\to \fH_n, \quad Z\mapsto(AZ+B)\cdot(CZ+D)^{-1}, 
\]
which reduces to the usual action $\mathrm{SL}(2,\mathbb{R})\cong \mathrm{Sp}(2,\mathbb{R}) \curvearrowright \BH^2$ when setting $n=1$.

In the following, we show that the stabilizer $\Stab_{\fH_n}(iI_n)\subset \mathrm{Sp}(2n,\BR)$ of $iI_n\in \fH_n$ is 
$\mathrm{Sp}(2n,\BR)\cap \mathrm{O}(2n,\BR)$, which is isomorphic to $U(n)$ via a specified isomorphism. 

The complex vector space $\BC^n$ can be considered as a real vector space $\BR^{2n}$ via the map $\BR^{2n}\to \BC^n$, $(x_1,\dots, x_n,y_1,\dots, y_n)\mapsto (x_1+iy_1, \dots, x_n+iy_n)$. Using this identification, the set $M_n(\BC)$ is considered as a subset of $ M_{2n}(\BR)$, and multiplication by $i\in \BC$ becomes the real $2n \times 2n$  matrix 
\[
J:=\left(
\begin{array}{cc}
 0 & -I_n \\
I_n & 0 \\
\end{array}
\right)
\in M_{2n}(\BR).  
\]
A real  $2n \times 2n$  matrix $E\in M_{2n}(\BR)$ is $\BC$-linear iff it commutes with $J$, iff $E$ is of the form
\[
E=\left(
\begin{array}{cc}
 A & -B \\
 B & A \\
\end{array}
\right)
\]
where $A, B\in M_{n}(\BR)$. Such a matrix   $E\in M_{n}(\BC) \subset M_{2n}(\BR)$   is identified with  $A+iB\in  M_{n}(\BC)$. 

\begin{lem}[Proof taken from website \cite{web}]
Using the inclusion $M_n(\BC) \subset M_{2n}(\BR)$ above, $\mathrm{U}(n)$ is identified with  $\mathrm{Sp}(2n,\BR)\cap \mathrm{O}(2n,\BR)$. 
\end{lem}

\begin{proof}

The matrix $A+iB\in  M_{n}(\BC)$ ($A, B\in M_{n}(\BR)$) is unitary iff 
\[
A^TA+B^TB=I_n \text{ and } A^TB=B^TA. 
\]

We have $\mathrm{U}(n) \subset \mathrm{Sp}(2n,\BR) $ because 
\[
\left(
\begin{array}{cc}
 A^T & B^T \\
 -B^T & A^T \\
\end{array}
\right)
\begin{pmatrix}
0 & I_n \\
-I_n & 0 
\end{pmatrix}
\left(
\begin{array}{cc}
 A & -B \\
 B & A \\
\end{array}
\right)
=
\begin{pmatrix}
0 & I_n \\
-I_n & 0 
\end{pmatrix}. 
\]

Similarly, $\mathrm{U}(n) \subset\mathrm{O}(2n,\BR)$ because 
\[
\left(
\begin{array}{cc}
 A^T & B^T \\
 -B^T & A^T \\
\end{array}
\right)
\left(
\begin{array}{cc}
 A & -B \\
 B & A \\
\end{array}
\right)
=
\begin{pmatrix}
I_n & 0 \\
0 & I_n 
\end{pmatrix}. 
\]

Conversely, take $E \in \mathrm{Sp}(2n,\BR)\cap \mathrm{O}(2n,\BR)$. We have $E^TE=I_{2n}$ and $E^T \Omega E=\Omega$ (or equivalently, $E^T J E=J$). It follows that
$
JE=EE^T J E=EJ$, hence $E$ is $\BC$-linear, having the form  $E=\left(
\begin{array}{cc}
 A & -B \\
 B & A \\
\end{array}
\right)$. Expanding $E^TE=I_{2n}$, we see that $A+iB$ is unitary. 
\end{proof}

%stabilizer

\begin{lem}
The stabilizer $\Stab_{\fH_n}(iI_n)\subset \mathrm{Sp}(2n,\BR)$ of $iI_n\in \fH_n$ is the group $\mathrm{U}(n)$, embedded in $\mathrm{Sp}(2n,\BR)$ as the subgroup $\mathrm{Sp}(2n,\BR)\cap \mathrm{O}(2n,\BR)$: 
\[
\mathrm{U}(n)\stackrel{\sim}{\to}\mathrm{Sp}(2n,\BR)\cap \mathrm{O}(2n,\BR) \subset   \mathrm{Sp}(2n,\BR), 
\]
\[
A+iB\mapsto  \left(
\begin{array}{cc}
 A & -B \\
 B & A \\
\end{array}
\right), \quad A, B\in M_{n}(\BR). 
\]
\label{lem:stabilizerembed}
\end{lem}

\begin{proof}
Assume that $E= \left(
\begin{array}{cc}
 A & C \\
 B & D \\
\end{array}
\right) 
\in
\mathrm{Sp}(2n,\BR)$
fixes $iI_n$. This implies 
$
(Ai+C)(Bi+D)^{-1}=iI_n$, 
$
(Ai+C)=i(Bi+D)$, 
and hence $A=D$, $C=-B$, $E=\left(
\begin{array}{cc}
 A & -B \\
 B & A \\
\end{array}
\right)\in
\mathrm{Sp}(2n,\BR)$, which  means
\[
\left(
\begin{array}{cc}
 A^T & B^T \\
 -B^T & A^T \\
\end{array}
\right)
\begin{pmatrix}
0 & I_n \\
-I_n & 0 
\end{pmatrix}
\left(
\begin{array}{cc}
 A & -B \\
 B & A \\
\end{array}
\right)
=
\begin{pmatrix}
0 & I_n \\
-I_n & 0 
\end{pmatrix}, 
\]
\[
\left(
\begin{array}{cc}
-B^TA+A^TB  & B^TB+A^TA \\
-A^TA-B^TB & A^TB-B^TA \\
\end{array}
\right)
=
\begin{pmatrix}
0 & I_n \\
-I_n & 0 
\end{pmatrix}. 
\]
It follows that the matrix $A+iB\in  M_{n}(\BC)$  is unitary. 

Conversely, given $A+iB\in \mathrm{U}(n)$ ($A, B\in M_{n}(\BR)$), we have $A^TA+B^TB=I_n$ and $ A^TB=B^TA$, and thus the equalities above hold. Then $\left(
\begin{array}{cc}
 A & -B \\
 B & A \\
\end{array}
\right)
\in
\mathrm{Sp}(2n,\BR)$ and 
its action sends $iI_n\in \fH_n$ to 
$
(Ai-B)(Bi+A)^{-1}=i(A+iB)(Bi+A)^{-1}=iI_n$. 
\end{proof}

The lemma above is used in Lemma \ref{lem:phi} in Subsection \ref{subsec:3.3}.

\subsection{The model space of Seifert geometry}

Consider the universal covering group $\sltrt$ of $\sltr$, acting on the left of the contractible space $\sltrt$ by left multiplication. Let $k:\BR\to\sltr$ be the homomorphism 
\[
k(r)=\left(\begin{array}{cc}\cos(\pi r)&-\sin(\pi r)\\\sin(\pi r)&\cos(\pi r)\end{array}\right)
\]
and let $\tilde{k}:\BR\to \sltrt$ be its lift. 
The center $Z(\sltrt)$ of $\sltrt$ is  the image of $\BZ$ under $\tilde{k}$.

Define $\sltrtxzr$ to be the central extension
\[
\sltrtxzr:=\sltrt\times \BR /\{(\widetilde{k}(n),-n):n\in\mathbb{Z}\}
\]
Note that in \cite{volofseirep} the group operation of $\BR$ is written additively, while in our model  we will write it multiplicatively. The image of $(g,s)\in \sltrt\times \BR$ in  $\sltrtxzr$ is denoted by $g[s]$, and we have 
\[
g[s]g^\prime[s^\prime]=(gg^\prime)[s+s^\prime], 
\]
\[
g\widetilde{k}(n)[s]=g[s+n], \forall n\in\mathbb{Z}. 
\]

The natural maps  $\sltrt\to \sltrtxzr, g\mapsto g[0]$ and $\mathbb{R}\to \sltrtxzr , s\mapsto \id[s]$  are embeddings. 
There is a short exact sequence
\[
0\to\mathbb{R}\to \sltrtxzr \to\mathrm{PSL}(2,\mathbb{R})\to 0. 
\]

Define the extended action $\sltrtxzr\curvearrowright \sltrt$ by 
\[
g[r]\cdot h=gh\widetilde{k}(r)
\]
for all $g[ r] \in \sltrtxzr$ and $h\in \sltrt$.  
This action is differentiable, transitive, and proper. It is proved in \cite[pp.~464-465]{scott}, that (after endowing $\sltrt$ with a suitable Riemannian metric) this action realizes $\sltrtxzr$ as the identity component of $\Isom(\sltrt)$, and $\Isom(\sltrt)$ has only  two components.

In a similar fashion with $\sltr$ replaced by $\spfr$ and more generally $\Sptnr$, we obtain first the left multiplication of $\spfrt$ on its corresponding homogeneous space $ \spfrt/\sutwo$ (Theorem \ref{thm:models}). The center of $\Sptnrt$ is described in Lemma \ref{lem:center}. In Subsection \ref{subsec:3.4}, we construct a similar central extension $\gxzghrh$ of $\CG:=\Sptnrt$, acting on the  model space $ \Sptnrt/\Sun$, and it turns out that this action is the action of $\gxzghrh$ on $\gxzghrh$ modulo its maximal compact subgroup (Proposition \ref{prop:extendedact}, Lemma \ref{lem:maximalcompactmodel}). It follows that there exists an invariant Riemannian metric. However, in this paper we are not going to take one and discuss the full isometry group. It might be an interesting question to determine what the full isometry group is.

\subsection{Seifert fiber spaces and a volume formula}

Seifert fiber spaces (also known as  Seifert manifolds) form six of Thurston's eight geometries. The $\sltrt$-geometry (also known as Seifert geometry), is one of the six geometries. A detailed treatment of the geometries and Seifert manifolds is given in \cite{martelli}. 

Our definition of Siegel--Seifert subgroups (Definition \ref{def:siegelseifert}) is motivated by \cite[Theorem~4.15]{scott} and its proof. Let $\Gamma$ be a discrete group of isometries of $\sltrt$ acting freely and
with quotient $M$. Then the foliation of $\sltrt$ descends to a foliation of $M$. 
We have exact sequences
\[
0\to\mathbb{R}\to\mathrm{Isom}\left(\sltrt\right)\xrightarrow{p}\mathrm{Isom}(\BH^2)\to 0, 
\]
\[
0\to \Gamma\cap\mathbb{R} \to \Gamma\to p(\Gamma)\to 0. 
\]
As $\Gamma$ is discrete, $\Gamma\cap\mathbb{R}$ is infinite cyclic or trivial. If $\Gamma\cap\mathbb{R}$ is infinite cyclic, we have that the foliation of $M$ is a Seifert fibration, and in particular a foliation into circles, fibering over a $2$-dimensional orbifold. 
Similar exact sequences hold for $ \CG\times_{Z(\CG)}\CH\CR/\CH$ (Proposition \ref{prop:exactseq}). 

%Similarly, for any subgroup $\Gamma \subset  \CG\times_{Z(\CG)}\CH\CR/\CH$
%there are exact sequences (Proposition \ref{prop:exactseq})
%\[
%0\to \theta(\{\mathrm{id}\}\times \CH\CR/\CH) \to \CG\times_{Z(\CG)}\CH\CR/\CH \xrightarrow{\eta} \CG/Z(\CG) \to 0, 
%\]
%\[
%0\to \Gamma\cap\ker \eta   \to \Gamma   \to \eta(\Gamma)\to 0. 
%\]
%The definition of a Siegel--Seifert subgroup is given in Definition \ref{def:siegelseifert}. 

A volume formula for $\BH^2\times \BR$-geometry and $\sltrt$-geometry is given in the proof of   \cite[Corollary~4.7.3]{thurston}. For any discrete group $\Gamma$ of isometries of $X=\BH^2\times \BR$ or $X=\sltrt$, if $p(\Gamma)$ is a discrete subgroup of $\mathrm{Isom}(\BH^2)$, we have 
\[
\mathrm{vol}(X/\Gamma)=\mathrm{area}(\mathbb{H}^{2}/p(\Gamma))\operatorname{length}(\mathbb{R}/(\Gamma\cap\ker p)). 
\]
This motivates our Theorem \ref{thm:vol} on the volume formula for ``Siegel--Seifert'' manifolds. 
It also suggests that the invariant measure in our case could be a ``product measure'', the meaning of which is made clear in Subsections \ref{subsec:5.2} and \ref{subsec:5.3}.

\section{Construction of the model}
\label{sec:3}

Let $\spfr\curvearrowright \fH_2$ be the Siegel upper half-space on which $\spfr$ acts. 
In this section we construct an action $\spfr\curvearrowright \fH_2\times S^1$ from the normal bundle of an equivariant embedding of $\fH_2$ into a Grassmannian manifold. This model lifts to $\spfrt\curvearrowright \fH_2\times \BR$, which turns out to be the homogeneous space $\spfrt\curvearrowright \spfrt/\sutwo$, where $\sutwo$ is a specified maximal compact subgroup in $\spfrt$. 
Finally, in Subsection \ref{subsec:3.4}, we consider general $n$ and  perform a central extension on $\Sptnrt$ to obtain a group $\gxzghrh$ acting on $\CX=\Sptnrt/\Sun$, realizing the model space $\CX=\Sptnrt/\Sun$ as $\gxzghrh$ modulo its maximal compact subgroup.

%
%
%Recall some basic definitions: 
%
%
%
%
%Let $I_n$ be the $n\times n$ identity matrix. Let
%\[
%\Omega= \begin{pmatrix}
%0 & I_n \\
%-I_n & 0 
%\end{pmatrix}.   
%\]
%
%The symplectic group $\mathrm{Sp}(2n,\BR)$ is defined to be 
%$$\mathrm{Sp}(2n,\BR)=\{M\in M_{2n\times2n}(\BR) | M^{\mathrm{T}}\Omega M=\Omega\}. $$
%
%
%The \textbf{Siegel upper half-space} $\fH_n$ is defined to be the set of $n\times n$ symmetric matrices over $\BC$ such that the imaginary part is positive definite. Writing $Z\in \fH_n$ as $Z=X+iY$ for real matrices $X,Y$, this means that $Y$ is positive definite. 
%
%The action $\mathrm{Sp}(2n,\BR)\curvearrowright \fH_n$ is given as follows. For $M= \begin{pmatrix}
%A & B \\
%C & D 
%\end{pmatrix} \in  \mathrm{Sp}(2n,\BR)$ decomposed into $n\times n$ matrices $A,B,C,D$, the action is 
%\[
%\fH_n\to \fH_n, \quad Z\mapsto(AZ+B)\cdot(CZ+D)^{-1}. 
%\]

\subsection{Embedding into a Grassmannian manifold}
\label{subsec:3.1}

The following  equivariant embedding from $\fH_n$ to a Grassmannian manifold is taken from   \cite[p.~94]{hulek}. 
Let $P_n$ be the set of all $(2n \times n)$ complex matrices of rank $n$ divided out by the equivalence
relation 
$$\left.\left(\begin{array}{c}M_1\\M_2\end{array}\right.\right)\sim\left(\begin{array}{c}M_1M\\M_2M\end{array}\right)\text{ for any }M\in\mathrm{GL}(n,\mathbb{C}).$$
Define an action  $\mathrm{Sp}(2n,\BR)\curvearrowright P_n$, 
\[
\left(\begin{array}{cc}A&B\\C&D\end{array}\right)\left[\begin{array}{c}M_{1}\\M_{2}\end{array}\right]=\left[\begin{array}{c}AM_{1}+BM_{2}\\CM_{1}+DM_{2}\end{array}\right], 
\]
where $[\quad]$ denotes equivalence classes in $P$. 

Consider the embedding 
\[
Z\in \fH_n\mapsto   \left[\begin{array}{c}Z\\I_{n}\end{array}\right]
\]
from $\fH_n$ into (a coordinate chart of) $P_n$. Since
\[
\left(\begin{array}{cc}A&B\\C&D\end{array}\right)\left[\begin{array}{c}Z\\I_{n}\end{array}\right]=\left[\begin{array}{c}AZ+B\\CZ+D\end{array}\right]=\left[\begin{array}{c}(AZ+B)(CZ+D)^{-1}\\I_{n}\end{array}\right], 
\]
the subspace $\fH_n\subset P_n$ is invariant under the action of $\mathrm{Sp}(2n,\BR)$, and $\mathrm{Sp}(2n,\BR)\curvearrowright P_n$ restricts to the usual action $\mathrm{Sp}(2n,\BR)\curvearrowright \fH_n$.

\subsection{The normal bundle}
\label{subsec:3.2}

Viewing in the coordinate chart of $P_n$ where the lower $n\times n$ matrix is invertible, any element can be uniquely written in the form 
$\left[\begin{array}{c}Z\\I_{n}\end{array}\right]$ where $Z$ can be any $n\times n$ complex matrix. The subset $\fH_n\subset P_n$ in the chart  is an open subset of a closed subset (indeed, an open subset of the hyperplane of symmetric matrices), and hence we can talk about the normal bundle.

Throughout this subsection, working in this chart, we consider only the case $n=2$ and identify $\left[\begin{array}{c}Z\\I_{2}\end{array}\right]$ with $Z\in M_{2\times 2}(\BC)$ for simplicity.

We explicitly define \textit{the} normal bundle, as a subbundle of $TP_2|_{\fH_2}$ of real rank $2$, by two global sections 
\[
vec_1=\begin{pmatrix}
0 & 1 \\
-1 & 0 
\end{pmatrix}, 
vec_2=\begin{pmatrix}
0 & i \\
-i & 0 
\end{pmatrix} =i\cdot vec_1, 
\]
considered as global sections of the tangent bundle of $M_{2\times 2}(\BC)$, restricted to $\fH_2$. 

\begin{prop}
The normal bundle is invariant under the tangent map of $\spfr\curvearrowright P_2$. 
\end{prop}

\begin{proof}
It is proved in \cite[Proposition~4.10]{follandsp}  that
 $\operatorname{Sp}(2n,\mathbb{R})$ is generated by $D(n)\cup N(n)\cup\{\Omega\}$, where
$$D(n)=\left\{\begin{pmatrix}A&0\\0&(A^T)^{-1}\end{pmatrix} \bigg|\: A\in\operatorname{GL}(n,\mathbb{R})\right\}, $$
$$N(n)=\left\{\begin{pmatrix}I_n&B\\0&I_n\end{pmatrix}\bigg|\:B \text{ is symmetric}\right\}. $$

It suffices to verify for these generators that the tangent vectors $vec_1,vec_2$ are mapped into $\mathrm{span}\{vec_1,vec_2\}$, which is done by brute force using Mathematica \cite{mathematica}.

%%%%%%%%%%%%%%%%%%%%%%%%%%%%%%%%%%%%%%%%%
%
%start of mathematica \text{ substitution
%
%%%%%%%%%%%%%%%%%%%%%%%%%%%%%%%%%%%%%%%%%

For arbitrary 
\[
X+iY=
\left(
\begin{array}{cc}
 x_1 & x_2 \\
 x_2 & x_3 \\
\end{array}
\right)+i \left(
\begin{array}{cc}
 y_1 & y_2 \\
 y_2 & y_3 \\
\end{array}
\right), 
\]
take the curve
\[
s\mapsto 
s \left(
\begin{array}{cc}
 0 & 1 \\
 -1 & 0 \\
\end{array}
\right)+\left(
\begin{array}{cc}
 x_1 & x_2 \\
 x_2 & x_3 \\
\end{array}
\right)+i \left(
\begin{array}{cc}
 y_1 & y_2 \\
 y_2 & y_3 \\
\end{array}
\right)
\]
representing the tangent vector $vec_1$. Now apply the action, take derivative with respect to $s$, and set $s=0$. 

For the action of $\Omega$, the image of $vec_1$ is 
\[
\left(
\begin{array}{cc}
 0 & \frac{1}{x_1 x_3+i x_1 y_3-x_2^2-2 i x_2 y_2+i x_3 y_1-y_1 y_3+y_2^2} \\
 -\frac{1}{x_1 x_3+i x_1 y_3-x_2^2-2 i x_2 y_2+i x_3 y_1-y_1 y_3+y_2^2} & 0 \\
\end{array}
\right)
\]

For the action of $M=\begin{pmatrix}
A & B \\
C & D 
\end{pmatrix}$ where $(A,B,C,D)$ is 

\[
\left(
\left(
\begin{array}{cc}
 a_1 & a_2 \\
 a_3 & a_4 \\
\end{array}
\right),\left(
\begin{array}{cc}
 0 & 0 \\
 0 & 0 \\
\end{array}
\right),\left(
\begin{array}{cc}
 0 & 0 \\
 0 & 0 \\
\end{array}
\right),\left(
\begin{array}{cc}
 \frac{a_4}{a_1 a_4-a_2 a_3} & -\frac{a_3}{a_1 a_4-a_2 a_3} \\
 -\frac{a_2}{a_1 a_4-a_2 a_3} & \frac{a_1}{a_1 a_4-a_2 a_3} \\
\end{array}
\right)
\right), 
\]
 the image of $vec_1$ is 
\[
\left(
\begin{array}{cc}
 0 & a_1 a_4-a_2 a_3 \\
 a_2 a_3-a_1 a_4 & 0 \\
\end{array}
\right). 
\]

For the action of $M=\begin{pmatrix}
A & B \\
C & D 
\end{pmatrix}$ where $(A,B,C,D)$ is 

\[
\left(\left(
\begin{array}{cc}
 1 & 0 \\
 0 & 1 \\
\end{array}
\right),\left(
\begin{array}{cc}
 b_1 & b_2 \\
 b_3 & b_4 \\
\end{array}
\right),\left(
\begin{array}{cc}
 0 & 0 \\
 0 & 0 \\
\end{array}
\right),\left(
\begin{array}{cc}
 1 & 0 \\
 0 & 1 \\
\end{array}
\right)\right), 
\]
 the image of $vec_1$ is 
\[
\left(
\begin{array}{cc}
 0 & 1 \\
 -1 & 0 \\
\end{array}
\right). 
\]

In any of the three cases, the image of $vec_2$ is $i$ times the image of $vec_1$. 
\end{proof}

%%%%%%%%%%%%%%%%%%%%%%%%%%%%%%%%%%%%%%%
%
%end of mathematica \text{ substitution
%
%%%%%%%%%%%%%%%%%%%%%%%%%%%%%%%%%%%%%%%

\subsection{Arriving at $\spfrt\curvearrowright\spfrt/\sutwo$}
\label{subsec:3.3}

Observe that we obtained an action of $\spfr$ on the trivial normal bundle $\fH_2\times \BC$. When restricted to a single fiber $\BC$, the action $\BC\to \BC$ becomes multiplication by a number in $\BC^*\cong S^1\times \BR$. For $\spfr\curvearrowright\fH_2\times \BC$, we now forget the radial part and record the rotations only, producing an action 
\[
\spfr\curvearrowright \fH_2\times S^1. 
\]

This action can be lifted to the universal covering spaces
\[
\spfrt\curvearrowright \fH_2\times \BR. 
\]

For the details of lifting an action, see for example Section 9 of  Chapter I  in \cite{bredontg}. The rest of this subsection is devoted to the proof of
\begin{thm}
The actions $\spfr\curvearrowright \fH_2\times S^1$ and $\spfrt\curvearrowright \fH_2\times \BR$ are both transitive and have compact stabilizer isomorphic to $\sutwo$. Thus after fixing basepoints we have identifications
\[
\spfr\curvearrowright \fH_2\times S^1\cong \spfr\curvearrowright \spfr/\sutwo, 
\]
\[
\spfrt\curvearrowright \fH_2\times \BR\cong \spfrt\curvearrowright \spfrt/\sutwo. 
\]
\label{thm:models}
\end{thm}

\subsubsection{Preparations before lifting. }
First of all, we recall some basic facts about $\spfr$ and $\utwo$. In the Siegel upper half-space $\spfr\curvearrowright \fH_2$, the stabilizer $\Stab_{\fH_2}(iI_2)$ of $iI_2\in \fH_2$ is isomorphic to $\utwo$, and is the maximal compact subgroup of $\spfr$. Topologically, $\spfr$ is just $\utwo$ times a vector space. 

There is an exact sequence
\[
0\to \sutwo\to \utwo \to S^1\to 0
\]
where $\det:\utwo \to S^1$ is just the determinant, and is also a trivial principal $\sutwo$-bundle with a global section given by 
\[
S^1 \to \utwo, 
e^{i\theta}\mapsto  \begin{pmatrix}
e^{i\theta} & 0 \\
0 & 1 
\end{pmatrix}  . 
\]
This global section induces an isomorphism on fundamental groups. 

Topologically, the total space $\utwo$ of the fiber bundle is just homeomorphic to a product $S^3\times S^1$. Such a homeomorphism identifies the subgroup $\sutwo$ with the subspace $S^3\times \{*\}\subset S^3\times S^1$, and identifies the image of the global section with the subspace $\{*\} \times S^1\subset S^3\times S^1$.

The group $\Stab_{\fH_2}(iI_2)\cong \utwo$ fixes the basepoint $iI_2$ in terms of the action $\spfr\curvearrowright \fH_2$, and hence sends the fiber $S^1$ at $iI_2$ to itself in terms of the action $\spfr\curvearrowright \fH_2\times S^1$. This defines a group  homomorphism 
\[
\phi:\utwo\cong \Stab_{\fH_2}(iI_2)\to S^1
\]
where the $S^1$ on the right hand side acts on the fiber $S^1$. 

\begin{lem}
$\phi(M)=(\det (M))^{-1}$. 
\label{lem:phi}
\end{lem}

\begin{proof}
Recall 
\[
\mathrm{SU}(2)=\left\{\begin{pmatrix}\alpha&-\overline{\beta}\\\beta&\overline{\alpha}\end{pmatrix}\bigg|\:\alpha,\beta\in\mathbb{C},\left|\alpha\right|^2+\left|\beta\right|^2=1\right\}. 
\]
The embedding (see Lemma \ref{lem:stabilizerembed})
\[
\sutwo\to \utwo \cong \Stab_{\fH_2}(iI_2)\to \spfr
\]
 is given by
\[
\begin{pmatrix}
\alpha&-\overline{\beta}\\
\beta&\overline{\alpha}
\end{pmatrix}
=
\begin{pmatrix}
a_1+i a_2&-b_1+i b_2\\
b_1+i b_2&a_1-i a_2
\end{pmatrix}
\mapsto
 \begin{pmatrix}
a_1 & -b_1 & -a_2 & -b_2 \\
b_1 & a_1 & -b_2 & a_2 \\
a_2 & b_2 & a_1 & -b_1 \\
b_2 & -a_2 & b_1 & a_1 
\end{pmatrix}  , 
\]
where $\alpha=a_1+i a_2, \beta=b_1+i b_2$ and   $a_1, a_2, b_1, b_2\in \BR$. 
Direct computation shows that the action of this matrix fixes  $iI_2$ and is the identity on the fiber $S^1$ over $iI_2$. It follows that $\sutwo \subset \ker(\phi)$, and the homomorphism $\phi$ factors as 
\[
\phi:\utwo\to \utwo/\sutwo\cong S^1\to S^1. 
\]
The homomorphism $S^1\to S^1$ above is just $z\mapsto z^k$ for some $k$, and we need to determine which integer it is. 

Observe also that  (see Lemma \ref{lem:stabilizerembed})
\[
\begin{pmatrix}
e^{i\theta} & 0 \\
0 & 1 
\end{pmatrix}
=
\begin{pmatrix}
\cos\theta+i \sin\theta & 0 \\
0 & 1 
\end{pmatrix}
\in \utwo
\mapsto
 \begin{pmatrix}
\cos\theta & 0 & -\sin\theta & 0 \\
0 & 1 & 0 & 0 \\
\sin\theta & 0 & \cos\theta & 0 \\
0 & 0 & 0 & 1 
\end{pmatrix}  
\in \spfr. 
\]
Direct computation shows that the action of this matrix fixes $iI_2$ and sends $vec_1$ (identified with $1\in S^1\subset \BC$) in the fiber $S^1$ over $iI_2$ to the following combination of $vec_1$ and $vec_2$: 
\[
 \begin{pmatrix}
0 & \cos\theta-i \sin \theta \\
-\cos\theta+i \sin \theta & 0 
\end{pmatrix}. 
\]
Consequently the homomorphism $S^1\to S^1$ above is just $z\mapsto z^{-1}$. 
\end{proof}

Write $q_0$ for the basepoint in $\fH_2\times S^1$ given by $vec_1$ at $iI_2$. The fiber $S^1$ over $iI_2$ is identified with the circle $\{vec_1\cos\theta  +vec_2\sin \theta \}_{\theta\in \BR}$ in the fiber $\BC$ (of the normal bundle) over $iI_2$. 

\begin{cor}
The action $\spfr\curvearrowright \fH_2\times S^1$ is transitive and $\Stab_{\fH_2\times S^1}(q_0)\cong \sutwo$. 
\label{cor:transtabs1}
\end{cor}

\begin{proof}
From the matrix computation above, we see that the orbit of $q_0$ contains the fiber $S^1$ over $iI_2$. As  $\spfr\curvearrowright \fH_2$ is transitive, $\spfr\curvearrowright \fH_2\times S^1$ is transitive. 

An element in $\Stab_{\fH_2\times S^1}(q_0)$ must fix $iI_2$ and hence is in $\Stab_{\fH_2}(iI_2)$. By the lemma above, $\Stab_{\fH_2\times S^1}(q_0)\cong \sutwo$. 
\end{proof}

\subsubsection{Lifting to the universal cover. }

Let 
\[
\pi:\spfrt\to \spfr
\]
be the universal covering. 
Using homeomorphisms above, the universal cover $\spfrt$ of $\spfr$ is homeomorphic to $\utwot$ times a vector space. Inside $\utwot$ which is homeomorphic to $S^3\times \BR$, there are $\BZ$ copies of $S^3$, denoted by 
\[
\pi^{-1}(\sutwo)=\{S^3_n\}_{n\in \BZ}. 
\]
 Consecutive spheres $S^3_n$ and $S^3_{n+1}$ are connected by an arc, which is  a single lift of the specified section. Explicitly, $S^3_n \subset\spfrt$ is represented by the following paths in $\spfr$ starting from identity: such a path first winds around the specified section $n$ times, and then goes anywhere while staying in $\sutwo$.

Fix a basepoint $\tilde{q}_0\in \fH_2\times \BR$ with image $q_0$ in $\fH_2\times S^1$. The inverse image of $q_0$ is $\{\tilde{q}_n\}_{n\in \BZ}$, where $\tilde{q}_n$ is the endpoint of the lift of the curve $\theta\in [0,2\pi]\mapsto vec_1\cos (n\theta)  +vec_2\sin (n\theta )$ starting at $\tilde{q}_0$. 

\begin{lem}
$S^3_n=\{g\in \spfrt |g\cdot \tilde{q}_0=\tilde{q}_{-n}\}$. 
\end{lem}

\begin{proof}
$S^3_n \subset \{g\in \spfrt |g\cdot \tilde{q}_0=\tilde{q}_{-n}\}$: An element in $S^3_n$ is represented by a path $\pi(\gamma)$ in $\spfr$ described above. Take its lift $\gamma:[0,1]\to \spfrt$ starting from the identity, and then the element becomes $\gamma(1)$. 
The path $\pi(\gamma)$ acts on $q_0$ in the following way: first, $q_0$ is moved along the fiber over $iI_2$, in a way that is computed in the matrix calculation above, and returns to $q_0$; then, the point $q_0$ is fixed. When applying the lifted path $\gamma$ to $ \tilde{q}_0$, we obtain a lift in  $\fH_2\times\BR$ of the former curve in $\fH_2\times S^1$, starting from $\tilde{q}_0$. We then see that $\tilde{q}_0$ is moved to $\tilde{q}_{-n}$.

$S^3_n \supset \{g\in \spfrt |g\cdot \tilde{q}_0=\tilde{q}_{-n}\}$: 
For $g\in \spfrt$ such that $g\cdot \tilde{q}_0=\tilde{q}_{-n}$, $\pi(g)$ must send $q_0$ to $q_0$, and hence $\pi(g)$ is in the stabilizer of $q_0$. Thus $g$ is in $\pi^{-1}(\sutwo)=\{S^3_n\}_{n\in \BZ}$. Considering the action of $g$, it must be in the claimed copy of $S^3$. 
\end{proof}

\begin{cor}
$\Stab_{\fH_2\times \BR}(\tilde{q}_0)=S^3_0\cong \sutwo $. 
\label{cor:stabr}
\end{cor}

\subsubsection{The proof. }

\begin{proof}[Proof of Theorem \ref{thm:models}]
By Corollaries \ref{cor:transtabs1} and \ref{cor:stabr}, it remains to prove transitivity of the action $\spfrt\curvearrowright \fH_2\times \BR$. 

Take any $\tilde{p}\in \fH_2\times \BR$ with image $p\in \fH_2\times S^1$. Transitivity of $\spfr\curvearrowright \fH_2\times S^1$ implies that there exists $g\in\spfr$ such that $g\cdot q_0=p$. Take a curve $\gamma$ in $\spfr$ connecting  $\id$ and $g$, which gradually moves $q_0$ to $p$ along the curve $t\mapsto \gamma(t)\cdot q_0$. Consider the lift of this curve determined by $\tilde{p}$. The other endpoint of this lift must be some lift of $q_0$, i.e. $\tilde{q}_{n}$ for some $n$. It follows that the following path starting from $\id$ in $\spfr$, considered as a point in $\spfrt$, will send $\tilde{q}_{0}$ to $\tilde{p}$: the path first winds around the specified circle suitably many times, and then goes along $\gamma$ to $g$. 
\end{proof}

%ver2 starts here

\subsection{Central extension}
\label{subsec:3.4}

Inspired by the model above, from now on we consider more generally the model corresponding to $\Sptnr$, $n\geq 2$ instead of $\spfr$. 
In this subsection, we perform a central extension on $\Sptnrt$ to obtain a group $\gxzghrh$ acting on $\CX=\Sptnrt/\Sun$, realizing the model space $\CX=\Sptnrt/\Sun$ as $\gxzghrh$ modulo its maximal compact subgroup. This is basically Section 5 of \cite{dlsw} with slight modifications.

\subsubsection{Structure of $\Sptnrt$}

We recall some basic facts about $\Sptnr$ and also $\Un$, which is embedded in $\Sptnr$ as the maximal compact subgroup with corresponding homogeneous space the Siegel upper half-space $\Sptnr\curvearrowright \fH_n$. 
The stabilizer $\Stab_{\fH_n}(iI_n)$ of $iI_n\in \fH_n$ is isomorphic to $\Un$, and is the maximal compact subgroup of $\Sptnr$. Topologically, $\Sptnr$ is just $\Un$ times a vector space. 

There is an exact sequence
\[
0\to \Sun\to \Un \to S^1\to 0
\]
where $\det:\Un \to S^1$ is just the determinant, and is also a trivial principal $\Sun$-bundle with a global section given by 
\[
S^1 \to \Un, 
e^{i\theta}\mapsto   \begin{pmatrix}
e^{i\theta} &   &   &   \\
  & 1 &   &   \\
  &   & \ddots &   \\
  &   &   & 1 
\end{pmatrix}   . 
\]
This global section induces an isomorphism on fundamental groups. 

Topologically, the total space $\Un$ of the fiber bundle is just homeomorphic to a product $\Sun\times S^1$. Such a homeomorphism identifies the subgroup $\Sun$ with the subspace $\Sun\times \{*\}\subset \Sun\times S^1$, and identifies the image of the global section with the subspace $\{*\} \times S^1\subset \Sun\times S^1$.

Let 
\[
\pi:\Sptnrt\to \Sptnr
\]
be the universal covering. 
Using homeomorphisms above, the universal cover $\Sptnrt$ of $\Sptnr$ is homeomorphic to $\Unt$ times a vector space. Inside $\Unt$ which is homeomorphic to $\Sun\times \BR$, there are $\BZ$ topological  copies of $\Sun$. The copy containing identity is mentioned as \textit{the} subgroup $\Sun$.

Recall the center of $\Sptnr$ is $\{\pm\mathrm{id}\}$. 

\begin{lem}
The center of $\Sptnrt$ is $\pi^{-1}(\{\pm\mathrm{id}\})$. For $n $ even, $ \pi^{-1}(\{\pm\mathrm{id}\})\cong \BZ\oplus \BZ/2\BZ$, while for $n$ odd $\pi^{-1}(\{\pm\mathrm{id}\})\cong \BZ$. 
\label{lem:center}
\end{lem}

\begin{proof}
For any general group homomorphism $\phi:G\to H$, we have $\phi(Z(G))\subset Z(\phi(G))$, and hence $Z(G)\subset \phi^{-1}(Z(\phi(G)))$. 
Conversely, every discrete normal subgroup of a connected Lie
group is central.  Hence the center of $\Sptnrt$ is $\pi^{-1}(\{\pm\mathrm{id}\})$. 

There is  a   short exact sequence of discrete abelian groups
\[
0 \to  \ker\pi    \to    \pi^{-1}(\{\pm\mathrm{id}\})    \to   \{\pm\mathrm{id}\}  \to  0. 
\]
where $\ker\pi\cong \mathbb{Z}$, and $\pi^{-1}(\{\pm\mathrm{id}\})=Z(\Sptnrt)  $. 

For $n $ even, $-\mathrm{id}\in \mathrm{GL}(n,\BC)\subset\mathrm{GL}(2n, \BR)$ is in $\Sun\subset \Un\subset\Sptnr$. Passing to the universal cover, the specified subgroup $\Sun\subset \Unt$ contains  $-\mathrm{id}$, providing a splitting of the short exact sequence above, and hence in this case $ \pi^{-1}(\{\pm\mathrm{id}\})\cong \BZ\oplus \BZ/2\BZ$.

For $n$ odd $-\mathrm{id}\notin \Sun$. Take a curve connecting $\id$ and $-\id$ in $\Un$, representing an element in $\Unt$,  as follows: first connect $\id$ and the $n$-by-$n$ complex matrix 
$$ \begin{pmatrix}
1 &   &   &   \\
  & -1 &   &   \\
  &   & \ddots &   \\
  &   &   & -1 
\end{pmatrix}  $$
within $\Sun$, and then move along 
\[
t\in [0,\pi]\mapsto  \begin{pmatrix}
e^{i t} &   &   &   \\
  & -1 &   &   \\
  &   & \ddots &   \\
  &   &   & -1 
\end{pmatrix} .  
\]
Observe that the square of this element in $\Unt$ generates $\ker\pi$. It follows that in this case $\pi^{-1}(\{\pm\mathrm{id}\})\cong \BZ$ with a generator specified by the curve above. 
\end{proof}

%warning: for a proof of un as a maximal compact subgroup of sp see Introduction to Symplectic Topology MCDUFF SALAMON

\subsubsection{Structure of $\Unt$. }
We first recall the structure of  $\Unt$.

By \cite[Exercise~E6.4]{hm},  
for $n\geq 2, $ the Lie algebra $\mathfrak{u}(n)$ has imaginary scalar matrices as center $\mathfrak{z}$; the commutator subalgebra is the simple  Lie algebra
\[
\mathfrak{su}(n)=\{X\in\mathfrak{gl}(n,\mathbb{C})\mid X+X^*=0\text{ and }\mathrm{Tr} X=0\}\neq 0.
\]
By \cite[Theorem~6.4]{hm}  $\mathfrak{u}(n)$ splits as a direct sum
\[
\mathfrak{u}(n)= \mathfrak{z}\oplus \mathfrak{su}(n). 
\]

Let $Z(\Un)\cong\mathrm{U}(1)\cong S^1$ denote the center of $\Un$. By   \cite[Lemma~6.8]{hm},  the function $(z, s) \mapsto zs: Z(\Un)\times \Sun\to\Un$ is a surjective
morphism  whose kernel is isomorphic to $Z(\Un) \cap \Sun=  \{e^{2\pi ik/n}\cdot\mathrm{id}\mid k=0,\dots,n-1\} $. This is a covering. 

The universal cover $\Unt$ can be directly read off: $\Unt\cong \mathbb{R}\times \Sun$. The universal covering map is 
\[
\mathbb{R}\times \Sun\to Z(\Un)\times \Sun \to \Un, 
\]
inducing the identity map on the level of Lie algebra. 

The $\mathbb{R}\subset\Unt$ above corresponds to the center $\mathfrak{z}$. Explicitly, 
\[
\mathfrak{z}=\{ia\cdot\mathrm{id}\mid a\in \mathbb{R}\}, 
\]
\[
Z(\Un)=\{e^{ia}\cdot \mathrm{id}\mid a\in \mathbb{R}\}, 
\]
and the $\mathbb{R}\subset\Unt$ is the lift of  $Z(\Un)$.

As $a$ goes from $0$ to $2\pi$, $e^{ia}\cdot \mathrm{id}$ gives a curve representing $n$ times the generator of the fundamental group. This can be seen by noting that $\mathrm{det}:\Un\to S^1$ induces an isomorphism on fundamental groups.

The subgroup $\Sun\subset \Sptnrt$ is a maximal compact subgroup. This can be seen by noting that $\Sptnrt$ deformation retracts to $\Unt$, and hence also to $\Sun$. See also the proof of Lemma 
\ref{lem:maximalcompactmodel}.

%We take the claimed abelian subgroup $R$ of $G$ to be the $\mathbb{R}$ in the decomposition $\utwot\cong \mathbb{R}\times \sutwo$. 
\begin{lem-defn}[Modified Proposition 5.2 in \cite{dlsw}]
Let $\CG:=\Sptnrt$, and  $\CH:=\Sun$ be the  maximal compact subgroup of $\CG$. The center $Z(\CG)$ is described in Lemma \ref{lem:center}.  
Take the  abelian subgroup $\CR$ of $\CG$ to be the $\mathbb{R}$ in the decomposition $\Unt\cong \mathbb{R}\times \Sun$  above. 
Then 

(i) The subgroup of $\CG$ generated by $\CH=\Sun$ and $\CR$ is a direct product $\CH\CR=\Unt$.

(ii) The intersection of $\CR$ and $Z(\CG)$ is a lattice of $\CR$ and a subgroup of $Z(\CG)$ of index $n$. 

(iii) The center $Z(\CG)$ of $\CG$ is contained by $\CH\CR$. 

(iv)  The quotient $\CH\CR/Z(\CG)\cong\Un /\{\pm \mathrm{id}\}$ is embedded as a maximal compact subgroup in $\CG/Z(\CG)$. 
\label{lemdef:ghr}
\end{lem-defn}

\begin{proof}
(ii):   Number the elements in $\pi^{-1}(\{\pm\mathrm{id}\})$ as follows. Such an element is represented by a curve in $\Un$ starting from identity. We say that the element is the \textit{$k$-th lift}, $k\in \frac{1}{2}\BZ$,  if the element is represented by a curve whose image in $S^1$ (via $\det:\Un\to S^1$) winds around the circle $k$ times counterclockwise. 

For $n$ odd, note that as the parameter for $\CR$ increases from zero, the trajectory hits successively the following points in $\pi^{-1}(\{\pm\mathrm{id}\})$: $0$-th lift of $\id$, $\frac{n}{2}$-th lift of $-\id$, $n$-th lift of $\id$, $\frac{3n}{2}$-th lift of $-\id$, and so on. In this case points in $\pi^{-1}(\{\mathrm{id}\})$ are numbered by $\BZ$, while points in $\pi^{-1}(\{-\mathrm{id}\})$ are numbered by $\BZ+\frac{1}{2}$.  

For $n$ even, the trajectory hits successively $0$-th lift of $\id$, $\frac{n}{2}$-th lift of $-\id$, $n$-th lift of $\id$, $\frac{3n}{2}$-th lift of $-\id$, and so on. In this case points in $\pi^{-1}(\{\mathrm{id}\})$ are numbered by $\BZ$, and points in $\pi^{-1}(\{-\mathrm{id}\})$ are also numbered by $\BZ$. The result follows.

The proof of (iv) is given in the following lemma. 
\end{proof}

\begin{lem}
The maximal compact subgroup of $\CG/Z(\CG)\cong\Psptnr$ is $\Un /\{\pm \mathrm{id}\}$. 
\end{lem}

\begin{proof}
The group $\Un /\{\pm \mathrm{id}\}$ is compact. If it is not maximal, pick a maximal one $K_0$ containing it. Taking inverse image with respect to the double covering $p: \Sptnr\to \Psptnr$, we obtain a strict inclusion $\Un \subsetneq p^{-1}(K_0)$. But $\Un $ is already maximal. 
\end{proof}

%1.3.3

\subsubsection{The central extension
$\CG\times_{Z(\CG)}\CH\CR/\CH$ and its action on $\CX$. }

For $n$ even, the map $\iota:Z(\CG)\cong   \mathbb{Z}\oplus \mathbb{Z}/2\mathbb{Z}       \to \CH\CR/\CH\cong \mathbb{R}$ is not an embedding. Its kernel is $Z(\CG)\cap \CH \cong \mathbb{Z}/2\mathbb{Z} $, consisting of the $0$-th lift of $-\id$ and the $0$-th lift of $\id$. 
For $n$ odd, the map $\iota:Z(\CG)\cong   \mathbb{Z}       \to \CH\CR/\CH\cong \mathbb{R}$ is injective. 

Together with the inclusion $i:Z(\CG)\to \CG$, we create the central extension: 

\begin{defn}
The central extension
$\CG\times_{Z(\CG)}\CH\CR/\CH$
is the quotient of $\CG\times \CH\CR/\CH$ by relations $(g,x)=(gz, x(\iota(z))^{-1})$, $z\in Z(\CG)$. That is to say, 
\[
\CG\times_{Z(\CG)}\CH\CR/\CH:=  \frac{\CG\times \CH\CR/\CH}{(i,(\iota)^{-1})Z(\CG)}. 
\]
\label{def:gxzghrh}
\end{defn}

The map $(i,(\iota)^{-1})$ is an injection. The subgroup $(i,(\iota)^{-1})Z(\CG)$ is discrete and central. 
Note that the group operation on  $\CH\CR/\CH\cong \mathbb{R}$ is written multiplicatively. We denote the elements of $\CH\CR/\CH$ as cosets $r\CH$ where $r\in \CR$. 

\begin{defn}
Let $\CX := \CG/\CH=\Sptnrt/\Sun$. Define an action
\[
\CG\times_{Z(\CG)}\CH\CR/\CH   \to    \mathrm{Diff}(\CX) 
\]
by
\[
%(g,rH)\cdot(xH)=(gxr^{-1})H.
(g,r\CH)\cdot(x\CH)=(gxr)\CH. 
\]
\label{def:extendedact}
\end{defn}
%This is well-defined as $(z, -\iota(z))\cdot(xH)=(zx\iota(z))H$, 
This is different from the one defined in \cite{dlsw}. 
This is well-defined as
\[
(g,r\CH)\cdot(xh\CH)=(gxhr)\CH=(gxrh)\CH=(gxr)\CH, \forall h\in \CH, 
\]
\[
((g_2,r_2\CH)(g_1,r_1\CH))\cdot(x\CH)
=(g_2g_1xr_2r_1)\CH
=(g_2g_1xr_1r_2)\CH
=(g_2,r_2\CH)\cdot((g_1,r_1\CH)\cdot(x\CH)), 
\]
\[
(z, (\iota(z))^{-1}\CH)\cdot(x\CH)
=zx (\iota(z))^{-1}\CH
=xz (\iota(z))^{-1}\CH
=x\CH, \forall z\in Z(\CG). 
\]

\subsubsection{Proof of Proposition \ref{prop:extendedact} }

We now show that the extended action 
\[
 \CG\times_{Z(\CG)}\CH\CR/\CH   \to    \mathrm{Diff}(\CX)=\mathrm{Diff}(\Sptnrt/\Sun) 
\]
 is transitive, and has compact stabilizer isomorphic to 
${\Un}/{\{\pm \id\}}$.

\begin{proof}[Proof of Proposition \ref{prop:extendedact}]

The composition
$
\CG \to \CG\times_{Z(\CG)}\CH\CR/\CH   \to    \mathrm{Diff}(\CX) 
$
is just the usual action on the coset space. 
Thus the action $ \CG\times_{Z(\CG)}\CH\CR/\CH   \to    \mathrm{Diff}(\CX) $ is transitive.

For the action $ \CG\times_{Z(\CG)}\CH\CR/\CH   \to    \mathrm{Diff}(\CX) $, consider the stabilizer of the basepoint, i.e. the coset $\CH$. 
\[
\begin{aligned}
&\mathrm{Stab}(\CH)\\
&=     \{(g,r\CH)\in  \CG\times_{Z(\CG)}\CH\CR/\CH  \mid     gr\in \CH                 \}\\
&=    \{(g,r\CH)\in  \CG\times_{Z(\CG)}\CH\CR/\CH  \mid   g=hs\in \CH\CR=\Unt, h\in \CH, s\in \CR,  gr=hsr\in \CH                 \}\\
&=\{(g,s^{-1}\CH)\in  \CG\times_{Z(\CG)}\CH\CR/\CH  \mid   g=hs\in \CH\CR=\Unt, h\in \CH, s\in \CR \}. \\
\end{aligned}
\]

Let $i:\CH\CR\to \CG$ be the embedding and $\iota: \CH\CR\to \CH\CR/\CH$ be the quotient map, extending the maps having the same notations above. Then 
$\mathrm{Stab}(\CH)$ is the image of the composition of  the injection $(i,(\iota)^{-1}):\CH\CR\to \CG\times \CH\CR/\CH$ and the quotient map $\CG\times \CH\CR/\CH\to \CG\times_{Z(\CG)}\CH\CR/\CH$. 
The kernel of the composition is $Z(\CG)$, and the image is 
\[
\mathrm{Stab}(\CH)\cong \frac{\CH\CR}{Z(\CG)}\cong \frac{\CH\CR/\ker \pi}{Z(\CG)/\ker \pi}\cong
\frac{\Un}{\{\pm \id\}}. 
\]
    It is compact.    
\end{proof}

\begin{lem}
The stabilizer $\mathrm{Stab}(\CH)$ is a maximal compact subgroup of $\CG\times_{Z(\CG)}\CH\CR/\CH$. 
\label{lem:maximalcompactmodel}
\end{lem}

\begin{proof}
There is a principal $\mathrm{Stab}(\CH)$-bundle
\[
\mathrm{Stab}(\CH)\to \CG\times_{Z(\CG)}\CH\CR/\CH \to \CX
\]
with base space $\CX$ contractible. Hence this is a trivial bundle, and $\mathrm{Stab}(\CH)$ and $\CG\times_{Z(\CG)}\CH\CR/\CH$ are homotopy equivalent. Both of the groups are connected. 

Assume that there is a strict inclusion $\mathrm{Stab}(\CH)\subsetneq K_1$ where $K_1$ is a maximal compact subgroup. It follows that $K_1$ is also homotopy equivalent to $\CG\times_{Z(\CG)}\CH\CR/\CH$. As $\mathrm{Stab}(\CH)$ and $ K_1$ are homotopy equivalent closed orientable manifolds, from homology groups we infer that they must have the same dimension. So $\mathrm{Stab}(\CH)$ is the component of $K_1$ containing identity. Given the strict inclusion $\mathrm{Stab}(\CH)\subsetneq K_1$, there are more components, but this contradicts the fact that $K_1$ is homotopy equivalent to $\CG\times_{Z(\CG)}\CH\CR/\CH$ which is connected. 
\end{proof}

\section{Fibering over the Siegel upper half-space}
\label{sec:4}
In this section, we show that 
the action of $\CG\times_{Z(\CG)}\CH\CR/\CH$ preserves the fibering structure of 
the trivial principal $\CH\CR/\CH$-bundle $\nu: \Sptnrt/\Sun \to  \Sptnrt/\Unt  $.  The induced action of $\CG\times_{Z(\CG)}\CH\CR/\CH$ on the base space coincides with the usual action of $\Sptnr$ on $\Sptnr/\Un$, and 
there is an exact sequence  of groups 
\[
0\to \theta(\{\mathrm{id}\}\times \CH\CR/\CH) \to \CG\times_{Z(\CG)}\CH\CR/\CH \to \CG/Z(\CG) \to 0, 
\]
where 
\[
\theta: \CG\times \CH\CR/\CH\to \CG\times_{Z(\CG)}\CH\CR/\CH
\]
is the quotient map,  $\theta(\{\mathrm{id}\}\times \CH\CR/\CH)\cong \mathbb{R}$,  and $\CG/Z(\CG)\cong \Psptnr$. 

As a consequence of the exact sequence, Subsection  \ref{subsec:4.3} discusses $\CG\times_{Z(\CG)}\CH\CR/\CH$ as a topological group extension.

\subsection{Fiber-preserving action}

Observe that the following diagram of spaces with $\spfrt$ action and equivariant maps is commutative: 
\[
% https://tikzcd.yichuanshen.de/#N4Igdg9gJgpgziAXAbVABwnAlgFyxMJZABgBpiBdUkANwEMAbAVxiRAAkB9AJgB1e8AW3gACfoLo4AFgCMZwAEoBfEEtLpMufIRQBGclVqMWbLnwFZhcEQGUAertXqQGbHgJFuB6vWatEHDxOGm7aRGS6hr4mAfxwaABmAE44APRxTDgA7hDBLpruOsj6kT7G-iBxiSnpvGhYdsAAtLpKGdm5aiFaHihepUZ+bFXJafyZOTiqhjBQAObwRKDJEIJIZCA4uYitzitrO9RbSNxdIPtIAMxH2wAsZxeItzdIAKwPSavrL4jXIAxYMAVKAQJgyBisahSGB0KBsSBAkBHOhYBjwgisD5fQ6bO7UAGIgIgsEQpEgaGw9GI5GoqmYvafA5eXFvfGA4Gg8GQ8kwuEBBHcnAotH8jHTJRAA
\begin{tikzcd}
\fH_2\times \mathbb{R} \arrow[r] \arrow[d, equal] & \fH_2\times S^1 \arrow[r] \arrow[d, equal] & \fH_2 \arrow[d, equal] \\
\spfrt/\sutwo \arrow[r]                                       & \spfrt/\pi^{-1}\sutwo \arrow[r]                        & \spfrt/\utwot                     
\end{tikzcd} 
\]
Here the vertical identifications are given by specifying basepoints $*$. The reason is that, for any point $a*\in \fH_2\times \mathbb{R}$, $a\in \spfrt$, the left square gives
\[
% https://tikzcd.yichuanshen.de/#N4Igdg9gJgpgziAXAbVABwnAlgFyxMJZABgBpiBdUkANwEMAbAVxiRDoCoQBfU9TXPkIoAjOSq1GLNpx58QGbHgJEyIifWatE7ADq64THAHcIc-kqFEx66puk66+tFgB6wALQju+wybPcEjBQAObwRKAAZgBOEAC2SGQgOGaIYsl0WAxscXRocCnmIDHxSOmFiADM1DiZ2Tq5+YW8UbEJiEkVAEw1dTl5BQHyJe09yanVGVn9TQEU3EA
\begin{tikzcd}
a* \arrow[r, maps to] \arrow[d, maps to] & a* \arrow[d, maps to] \\
a\sutwo \arrow[r, maps to]               & a\pi^{-1}\sutwo      
\end{tikzcd}
\]
which is commutative. The right square is similar.

This observation  leads us to consider the trivial fiber bundle $\nu: \spfrt/\sutwo \to  \spfrt/\utwot  $ with fiber $\mathbb{R}$, and more generally $\nu: \Sptnrt/\Sun=\CG/\CH \to  \Sptnrt/\Unt=\CG/\CH\CR  $. In general we have that 

\begin{lem}[\cite{lglagokv}, II.  Lie Transformation Groups, Chapter 2, Section 1, Lemma 1.1]
Let $G$ be a Lie group and $N \supset H$ two Lie subgroups, where $H$
is normal in $N$. Then the natural mapping $f : G /H \to G/N$ is the projection
of an (analytic) principal bundle, with structure group $N/H$, whose right action
on $G / H$ is given by the formula
\[
(gH)(nH)=gnH, (g\in G,n\in N). 
\]
\end{lem}

In particular, $\nu: \Sptnrt/\Sun=\CG/\CH \to  \Sptnrt/\Unt=\CG/\CH\CR  $ is a principal $\CH\CR/\CH$-bundle. Let
\[
\theta: \CG\times \CH\CR/\CH\to \CG\times_{Z(\CG)}\CH\CR/\CH
\]
be the quotient map. 
The group $\CH\CR/\CH\cong \mathbb{R}$ embeds in $ \CG\times_{Z(\CG)}\CH\CR/\CH$ as a central subgroup via the map  $r\CH\mapsto (\mathrm{id}, r\CH)$. We denote this subgroup by 
$
\theta(\{\mathrm{id}\}\times \CH\CR/\CH)
$.

\begin{lem}
The subgroup $ \theta(\{\mathrm{id}\}\times \CH\CR/\CH) \cong \CH\CR/\CH$  acts   on $\Sptnrt/\Sun=\CG/\CH$ as the structure group, i.e. translations along fiber direction. 
\end{lem}

\begin{proof}
For any point $s\CH\CR \in \CG/\CH\CR, s\in \CG$ in the base space, the fiber over it is 
\[
\{sr\CH\mid r\in \CR\}
\]
which is parametrized by $\CR$. 
A general element $(\mathrm{id}, t\CH)\in  \theta(\{\mathrm{id}\}\times \CH\CR/\CH) , t\in \CR$ sends $sr\CH$ to $srt\CH$. Note that this coincides with the action of the structure group above. 
\end{proof}

\begin{lem}
The action of $(g,r\CH)\in\CG\times_{Z(\CG)}\CH\CR/\CH$ is equivariant with respect to the action of the structure group. Each fiber is sent homeomorphically and equivariantly to some fiber. 
\end{lem}

\begin{proof}
The first assertion follows from the action of the central subgroup $\theta(\{\mathrm{id}\}\times \CH\CR/\CH)$ described in the lemma above. 

For $a_1\CH, a_2\CH \in \CG/\CH$ in the same fiber, i.e. $a_1\CH\CR=a_2\CH\CR$, consider their images
\[
(g,r\CH)\cdot  a_1\CH=  (ga_1r)\CH, 
\]
\[
(g,r\CH)\cdot  a_2\CH=  (ga_2r)\CH, 
\]
 under the action of $(g,r\CH)\in \CG\times_{Z(\CG)}\CH\CR/\CH$. As $r\in \CR\subset\CH\CR$,  
\[
(ga_1r)\CH\CR=(ga_1)\CH\CR
=(ga_2)\CH\CR
=(ga_2r)\CH\CR, 
\]
and hence the two images $(ga_1r)\CH$ and $(ga_2r)\CH$ are in the same fiber. 

Considering the inverse of $(g,r\CH)\in \CG\times_{Z(\CG)}\CH\CR/\CH$, we see that the action of $(g,r\CH)\in \CG\times_{Z(\CG)}\CH\CR/\CH$ sends an arbitrary fiber homeomorphically onto some image fiber. This homeomorphism is equivariant with respect to the action of the structure group, as for $t\in \CR$, 
\[
(g,r\CH)\cdot  at\CH=  (gatr)\CH=(gart)\CH. 
\]
\end{proof}

It follows that the action induces a bijection from the set of fibers to itself.

\subsection{An exact sequence}
\label{subsec:4.2}

We now prove the claimed exact sequence of groups
\[
0\to \theta(\{\mathrm{id}\}\times \CH\CR/\CH) \to \CG\times_{Z(\CG)}\CH\CR/\CH \to \CG/Z(\CG) \to 0, 
\]
where 
\[
\theta: \CG\times \CH\CR/\CH\to \CG\times_{Z(\CG)}\CH\CR/\CH
\]
is the quotient map,  $\theta(\{\mathrm{id}\}\times \CH\CR/\CH)\cong \mathbb{R}$,  and $\CG/Z(\CG)\cong \Psptnr$. 

\begin{proof}[Proof of Proposition \ref{prop:exactseq}]
Note that the subgroup of $\CG\times \CH\CR/\CH$ generated by $\{\mathrm{id}\}\times \CH\CR/\CH$ and 
$(i,(\iota)^{-1})Z(\CG)$ is the central subgroup $Z(\CG)\times \CH\CR/\CH$. Hence there are isomorphisms
\[
\frac{\CG\times_{Z(\CG)}\CH\CR/\CH}{\theta(\{\mathrm{id}\}\times \CH\CR/\CH)}\cong
\frac{(\CG\times \CH\CR/\CH)/((i,(\iota)^{-1})Z(\CG))}{(Z(\CG)\times \CH\CR/\CH)/((i,(\iota)^{-1})Z(\CG))}\cong
\frac{\CG\times \CH\CR/\CH}{Z(\CG)\times \CH\CR/\CH}\cong
\frac{\CG}{Z(\CG)}\cong
\Psptnr.  
\]
\end{proof}

The fiber-preserving action of $ \CG\times_{Z(\CG)}\CH\CR/\CH$ on $\Sptnrt/\Sun$ descends  to an action  on $\Sptnrt/\Unt$, on which the subgroup $ \theta(\{\mathrm{id}\}\times \CH\CR/\CH) $ acts trivially.  We obtain an action 
\[
\frac{\CG\times_{Z(\CG)}\CH\CR/\CH}{\theta(\{\mathrm{id}\}\times \CH\CR/\CH)}\cong \frac{\CG}{Z(\CG)}\cong
\Psptnr          \curvearrowright           \Sptnrt/\Unt \cong \Sptnr/\Un 
%\cong H_2
. 
\]

Observe that the usual  action of $\Sptnr$ on $\Sptnr/\Un$ factors through $\Psptnr$. This is because 
\[
-\mathrm{id} \cdot a\Un =a \cdot (-\mathrm{id}) \Un =a\Un. 
\]

\begin{lem}
The descended action 
\[
\frac{\CG\times_{Z(\CG)}\CH\CR/\CH}{\theta(\{\mathrm{id}\}\times \CH\CR/\CH)}
\curvearrowright 
 \Sptnrt/\Unt \cong \Sptnr/\Un 
%\cong H_2
\]
 coincides  with the usual action of $\Sptnr$ on $\Sptnr/\Un$. 
\label{lem:descendact}
\end{lem}

\begin{proof}
Take $a\in \Sptnrt$ and an element (denoted by $(g,r\CH)$) in $\frac{\CG\times_{Z(\CG)}\CH\CR/\CH}{\theta(\{\mathrm{id}\}\times \CH\CR/\CH)}$ represented by $(g,r\CH)\in \CG\times \CH\CR/\CH$. Then $(g,r\CH)$ sends $a\Unt\in \Sptnrt/\Unt$ to $gar\Unt=ga\Unt$. 

Under the isomorphism $\frac{\CG\times_{Z(\CG)}\CH\CR/\CH}{\theta(\{\mathrm{id}\}\times \CH\CR/\CH)}\cong \frac{\CG}{Z(\CG)}$, $(g,r\CH)$ corresponds to the image in $\CG/Z(\CG)$ of $g\in \CG$, which is still denoted by $g\in \CG/Z(\CG)$.

Recall $\pi: \Sptnrt\to \Sptnr$. Under the isomorphism $\Sptnrt/\Unt \cong \Sptnr/\Un $, $a\Unt$ is mapped to  $\pi(a) \Un$, while $ga\Unt$ is mapped to  $\pi(ga) \Un$. 
In this language the original action by  $(g,r\CH)$ is nothing but left multiplication by $g$. 
\end{proof}

%new subsection 4.3

\subsection{$\CG\times_{Z(\CG)}\CH\CR/\CH$ as a central extension}

\label{subsec:4.3}

In this subsection, we show that the exact sequence in Proposition \ref{prop:exactseq} is a topological group extension with a cross section. Such nontrivial extensions are classified by the second continuous cohomology, and can be done only to $6$ families of simple Lie groups, including the symplectic family considered in this paper.

A \textbf{topological group extension}, in the sense of \cite[Section~5]{hu}, is a pair $(E,\phi)$ where $E$ is a topological group containing the group $Q$ as a closed normal subgroup and $\phi$ is an open continuous homomorphism of $E$ onto $G$ with the subgroup $Q$ of $E$ as the kernel. The data is encoded in the exact sequence
\[
0\to Q \to E \stackrel{\phi}{\to} G\to 0. 
\]
A \textbf{cross section} of this topological group extension is a continuous map $u:G\to E$ such that $\phi u(x)=x$ for any $x\in G$.

It is proved in \cite{hu} (see also the survey \cite{stasheff}) that the second continuous cohomology $H^2_c(G;A)$ of a topological group $G$ with coefficients in a continuous $G$-module $A$, classifies topologically split group extensions, i.e. topological group extensions such that, considered as a principal bundle, the bundle is trivial.

\begin{prop}
The exact sequence in Proposition \ref{prop:exactseq} is a topological group extension with a cross section. 
\label{prop:topgrpext}
\end{prop}

\begin{proof}
By the proof of 
Proposition \ref{prop:exactseq} above,   the exact sequence 
\[
0\to 
%\theta(\{\mathrm{id}\}\times \CH\CR/\CH)
\BR \to \CG\times_{Z(\CG)}\CH\CR/\CH \stackrel{\eta}{\to} \CG/Z(\CG) \to 0, 
\]
is identified with 
\[
0\to \frac{Z(\CG)\times \CH\CR/\CH}{(i,(\iota)^{-1})Z(\CG)} \to \frac{\CG\times \CH\CR/\CH}{(i,(\iota)^{-1})Z(\CG)} \to \frac{\CG\times \CH\CR/\CH}{Z(\CG)\times \CH\CR/\CH} \to 0. 
\]
%{(\CG\times \CH\CR/\CH)/((i,(\iota)^{-1})Z(\CG))}{(Z(\CG)\times \CH\CR/\CH)/((i,(\iota)^{-1})Z(\CG))}

Here every group is given a Lie group structure (see \cite[Theorem~21.26]{lee}; the subgroup $(i,(\iota)^{-1})Z(\CG)$ is discrete and hence closed).  The subgroup $ \frac{Z(\CG)\times \CH\CR/\CH}{(i,(\iota)^{-1})Z(\CG)}$ of $\frac{\CG\times \CH\CR/\CH}{(i,(\iota)^{-1})Z(\CG)}$ is central, and is closed since  $Z(\CG)\times \CH\CR/\CH \subset \CG\times \CH\CR/\CH$ is closed. The quotient homomorphism $\frac{\CG\times \CH\CR/\CH}{(i,(\iota)^{-1})Z(\CG)} \to \frac{\CG\times \CH\CR/\CH}{Z(\CG)\times \CH\CR/\CH}$ is a continuous surjection, and openness follows from \cite[Lemma~21.1]{lee}. In conclusion, this exact sequence is a topological group extension. 

Topologically, considered as a principal $\BR$-bundle, this bundle is trivial. This is because the classifying space $B\BR$ is a point $\{*\}$, while the universal principal $\BR$-bundle is just $\BR\to \{*\}$. 
\end{proof}

For   a connected, non-compact, simple Lie group $G$ with finite center, $H^2_c(G;\BR)\neq 0$ if and only if $G$ is of Hermitian type (\cite[Lemme~1]{seconddegcontcoh}, see also \cite{thirddegcontcoh}). As a consequence,  the only possibilities  admitting (nontrivial) topologically split group extensions by $\BR$ (where the action by $G$ is trivial), are the $6$ families listed after \cite[Lemme~1]{seconddegcontcoh}, and we are performing extensions to one of these families. See also Appendices C.3 and C.4 of \cite{knapp}.  

\section{Volume formula}
\label{sec:5}

In this section, the volume formula of {Siegel--Seifert} (Definition \ref{def:siegelseifert}) closed manifolds is given.

Goetz's result \cite{goetz} summarized in Subsection \ref{subsec:5.2} shows that the invariant measure in our model is given by a ``product measure''  of the base and the fiber. After fixing such an invariant volume form, using the Chern--Gauss--Bonnet theorem \cite{cherngaussbonnet}, we derive the following volume formula (Theorem \ref{thm:vol}): 
For Siegel--Seifert subgroup $\Gamma \subset  \CG\times_{Z(\CG)}\CH\CR/\CH$, the volume of $\Gamma\backslash \CX$ is given by 
\[
\vol(\Gamma\backslash \CX)=\pm\vol((\Gamma\cap\ker\eta)\backslash  (\CH\CR/\CH)    )\cdot \chi(\eta(\Gamma)). 
\]

In particular, for Siegel--Seifert subgroup $\Gamma$ {arising from} (defined in Definition \ref{def:arisesfrom})  $\Psptnr$, the volume of $\Gamma\backslash \CX$ is given by 
\[
\vol(\Gamma\backslash \CX)=\pm \chi(\eta(\Gamma)). 
\]

\subsection{Siegel--Seifert subgroups}
\label{subsec:5.1}

Recall the definition of Siegel--Seifert subgroups (Definition \ref{def:siegelseifert}):

Let $\eta:  \CG\times_{Z(\CG)}\CH\CR/\CH \to \CG/Z(\CG)$ be the map in the exact sequence
\[
0\to \theta(\{\mathrm{id}\}\times \CH\CR/\CH) \to \CG\times_{Z(\CG)}\CH\CR/\CH \to \CG/Z(\CG) \to 0
\]
in Proposition \ref{prop:exactseq}. 

For any subgroup $\Gamma \subset  \CG\times_{Z(\CG)}\CH\CR/\CH$, there is an exact sequence
\[
0\to \Gamma\cap\ker \eta   \to \Gamma   \to \eta(\Gamma)\to 0. 
\]

A subgroup $\Gamma \subset  \CG\times_{Z(\CG)}\CH\CR/\CH$ is \textbf{Siegel--Seifert} if

(i) $\Gamma \subset  \CG\times_{Z(\CG)}\CH\CR/\CH$ is discrete and torsion-free, and the action $\Gamma\curvearrowright 
\Sptnrt/\Sun=\CX$ is cocompact and properly discontinuous, 

(ii) $\eta(\Gamma) \subset \CG/Z(\CG) $ is discrete and torsion-free, and the action $\eta(\Gamma)\curvearrowright \Sptnrt/\Unt =\fH_n$ is cocompact and properly discontinuous, and

(iii) $\Gamma\cap\ker\eta$ is infinite cyclic.

The corresponding manifold $\Gamma\backslash \mathcal{X}$ is called a \textbf{Siegel--Seifert} manifold. 
For a Siegel--Seifert manifold $\Gamma\backslash \mathcal{X}$, as $\Gamma\cap\ker\eta$ is infinite cyclic, the foliation on $\mathcal{X}$ into lines descends to a foliation into circles on $\Gamma\backslash \mathcal{X}$, fibering over the base manifold $\eta(\Gamma)\backslash\fH_n$.

These conditions imply that the action $\Gamma\curvearrowright 
\Sptnrt/\Sun=\CX$ and the action $\eta(\Gamma)\curvearrowright \Sptnrt/\Unt =\fH_n$ are both free. 
For concrete computations, we may also take into considerations the groups whose action is not free, or is just cofinite.

Pick a fundamental domain $D\subset \Sptnrt/\Unt $ for $\eta(\Gamma)$. Recall $\nu: \Sptnrt/\Sun \to  \Sptnrt/\Unt  $. 

\begin{prop}
A fundamental domain for Siegel--Seifert $\Gamma$ is given by a fundamental domain in $\nu^{-1}(D)$ for the action of $\Gamma\cap\ker\eta$. 
%To be precise, we take a product subset in $\nu^{-1}(D)\cong D\times \mathbb{R}$ of the form $D\times I$, where $I$ is an interval whose length is determined by $\Gamma\cap\ker\eta$. 
\label{prop:funddom}
\end{prop}

\begin{proof}
The group $\Gamma$ is partitioned into cosets $a(\Gamma\cap\ker\eta)$. The coset $\Gamma\cap\ker\eta$ acting on the claimed region fills $\nu^{-1}(D)$, while other cosets move $\nu^{-1}(D)$ to the inverse image of other copies of $D$ under $\nu$.  
\end{proof}

\subsection{Product measure}
\label{subsec:5.2}

This subsection is based on a paper \cite{goetz} by A. Goetz. It is proved here  that  the invariant measure in our model is given by a ``product measure''  of the base and the fiber.

Let $G$ be a Lie group with left-invariant Haar measure $\nu_G$ given by a top-dimensional left-invariant form $\omega_G$. Consider a principal $G$-bundle $B$ over base manifold $X$. 

Take a fiber $Y_x\curvearrowleft G$ over $x\in X$. Given a point $a\in Y_x$, there is a bijection $Y_x\to G, ag\mapsto g$, inducing from $\nu_G$ a measure on $Y_x$. For a different choice of point $ah\in Y_x$, the induced bijection is now $Y_x\to G, ar\mapsto h^{-1}r$, which gives the same measure on $Y_x$ since $\nu_G$ is left-invariant. Hence, on each fiber $Y_x$,  there is a measure $\nu_x$ induced from the left-invariant measure on $G$, which is actually given by a well-defined differential form. 

Take a measure $\mu$ represented by a form $\omega_\mu$ on $X$. For each trivialization $\psi_j:V_j\times G \to U\subset B$ of $B$ as a principal $G$-bundle, consider the product measure $\mu|_{V_j}\times \nu_G$ on $V_j\times G$ represented by the differential form $\pi_1^*\omega_\mu\wedge \pi_2^* \omega_G$, where $\pi_1, \pi_2$ are projections to factors. The diffeomorphism $\psi_j$ sends $\pi_1^*\omega_\mu\wedge \pi_2^* \omega_G$ to a form defined on $U\subset B$. These locally defined forms are compatible and patch together to give a globally defined form, and hence a \textbf{product measure} $\lambda$  on $B$. For any  Borel subset $Z \subset B$, setting $Z_x=Z\cap Y_x$, the product measure is given by 
\[
\lambda\left(Z\right)=\int_{{X}}\nu_{{x}}\left(Z_{{x}}\right)d\mu\left(x\right). 
\]

Consider a bundle map $h:B\to B'$ (between principal $G$-bundles) covering a diffeomorphism $\bar{h}:X\to X'$. Let $\mu'=\bar{h}_* \mu$ be the pushforward measure on $X'$. For any Borel set $Z'\subset B'$, 
\[
\lambda(h^{-1}(Z^{\prime}))=\int_X\nu_x([h^{-1}(Z^{\prime})]_x)d\mu(x)=\int_{X^{\prime}}\nu_{x^{\prime}}(Z_{x^{\prime}}^{\prime})d\mu^{\prime}(x^{\prime})=\lambda'(Z'). 
\]
In particular, 
\begin{lem}[\cite{goetz}]
If $h$ is a bundle map from  a principal $G$-bundle $B$ to itself and $\bar{h}$ preserves the measure $\mu$ on $X$, then $h$ preserves the product measure $\lambda$. 
\label{lem:productmeasure}
\end{lem}

\subsection{Invariant measure on $\CX$}
\label{subsec:5.3}

It follows from Lemma \ref{lem:productmeasure} and Lemma \ref{lem:descendact} that once an invariant measure on $\Sptnr\curvearrowright\Sptnr/\Un$ and a constant multiple of the standard measure on $\CH\CR/\CH\cong\BR$ are fixed, an invariant measure on $\CG\times_{Z(\CG)}\CH\CR/\CH\curvearrowright \CX$ is obtained. In this subsection, we fix the normalizations and obtain a volume formula.

Consider first $\CH\CR/\CH\cong\BR$.

%\begin{defn}
%For any subgroup $L\subset \Psptnr$, take its inverse image $L'\subset \Sptnrt$, send it to $L''\subset 
%\CG\times \CH\CR/\CH$ and then $L'''\subset \CG\times_{Z(\CG)}\CH\CR/\CH$. We say that $\Gamma:=L'''$ obtained in this way  
%\textbf{arises from} $\Psptnr$. 
%\label{def:arisesfrom}
%\end{defn}

\begin{defn}
We say that a subgroup $\Gamma\subset \CG\times_{Z(\CG)}\CH\CR/\CH$ \textbf{arises from} (a subgroup $L$ of) $\Psptnr$, denoted as $\Gamma=L^\uparrow$, if $\Gamma$ is constructed from $L$ as follows: For any subgroup $L\subset \Psptnr$, take its inverse image $L'\subset \Sptnrt$, send it to $L''\subset 
\CG\times \CH\CR/\CH$ and then to $L'''\subset \CG\times_{Z(\CG)}\CH\CR/\CH$ via the quotient map. Set $\Gamma=L^\uparrow:=L'''$. 
\label{def:arisesfrom}
\end{defn}

Note that $\eta(\Gamma)=L$. 
%Observe that 
%\[
%\begin{aligned}
%&L\subset \Psptnr \text{ is cocompact}\\
%\Leftrightarrow &L'\backslash \Sptnrt \text{ is compact}\\
%\Leftrightarrow &L'\backslash \Sptnrt/\Sun \text{ is compact}\\
%\Leftrightarrow &L'''\backslash \Sptnrt/\Sun \text{ is compact}\\
%\Leftrightarrow &L'''\subset \CG\times_{Z(\CG)}\CH\CR/\CH \text{ is cocompact}. 
%\end{aligned}
%\]

\begin{lem}
For Siegel--Seifert $\Gamma$ arising from $\Psptnr$ as above, $\Gamma \cap \ker \eta=\theta(Z(\CG))$ is an infinite cyclic subgroup of $\ker \eta \cong   \CH\CR/\CH     \cong \BR$
, where $\theta: \CG\times \CH\CR/\CH\to \CG\times_{Z(\CG)}\CH\CR/\CH$,  and subgroups of $\CG$ are also considered as subgroups of $ \CG\times \CH\CR/\CH$. 
\end{lem}

\begin{proof}
Since $Z(\CG)\times \CH\CR/\CH=\theta^{-1}(\ker \eta)$, 
\[
\Gamma \cap \ker \eta=\theta(L'')\cap \theta(Z(\CG)\times \CH\CR/\CH)=\theta(L''\cap (Z(\CG)\times \CH\CR/\CH))=\theta(Z(\CG)). 
\]
To show that $\theta(Z(\CG))$ is infinite cyclic, take the $n$-th lift $g_n$ of $\id$ (see the proof of Lemma-Definition \ref{lemdef:ghr}), which is in both $\CR$ and $Z(\CG)$, and is not the identity. Consider $\theta((g_n,g_n^{-1}))=\id\in\CG\times_{Z(\CG)}\CH\CR/\CH$ and $\theta((g_n,\id))\in \theta(Z(\CG))\subset\ker \eta \cong   \CH\CR/\CH     \cong \BR$. We have
\[
\theta((g_n,\id))=\theta((g_n,\id))\theta((g_n^{-1},g_n))=\theta((g_n,\id)(g_n^{-1},g_n))
=\theta((\id,g_n))\neq \id. 
\]
\end{proof}

\begin{conv}
The infinite cyclic subgroup $\theta(Z(\CG))$ in the lemma above has covolume $1$. 
\label{conv:volr}
\end{conv}

Now we turn to $\Sptnr\to\Psptnr\curvearrowright\Sptnr/\Un=\fH_n$. 

Take an invariant Riemannian metric on $\Sptnr/\Un=\fH_n$. By the well-known Chern--Gauss--Bonnet theorem \cite{cherngaussbonnet}, there is a differential form $\mathrm{Eu}$ constructed from the Pfaffian of the curvature form, satisfying the following property: for Siegel--Seifert $\Gamma$, the integration of $\mathrm{Eu}$ over the  closed orientable manifold $\eta(\Gamma)\backslash\fH_n$ is given by 
\[
\int_{\eta(\Gamma)\backslash\fH_n} \mathrm{Eu}=\chi(\eta(\Gamma)\backslash\fH_n), 
\]
the Euler characteristic of $\eta(\Gamma)\backslash\fH_n$.

\begin{conv}
Take $\pm\mathrm{Eu}$ to be the invariant volume form on $\Sptnr/\Un=\fH_n$. The  sign is introduced to make the resulting volume positive. 
\label{conv:voleu}
\end{conv}

As $\eta(\Gamma)\backslash\fH_n$ is an Eilenberg--Maclane space, we have
\[
\chi(\eta(\Gamma)\backslash\fH_n)=\chi(\eta(\Gamma)). 
\]

By Proposition \ref{prop:funddom} and the conventions above, we can prove Theorem \ref{thm:vol}.  

\begin{proof}[Proof of Theorem \ref{thm:vol}]
For any Siegel--Seifert subgroup $\Gamma \subset  \CG\times_{Z(\CG)}\CH\CR/\CH$, by Proposition \ref{prop:funddom}, a fundamental domain is given by a cylinder-shaped subset. More precisely, 
given any fundamental domain $D\subset \Sptnrt/\Unt $ for $\eta(\Gamma)$ action, a fundamental domain for Siegel--Seifert $\Gamma$ is given by a segment of the infinite cylinder $\nu^{-1}(D)$, whose length is determined by  the action of $\Gamma\cap\ker\eta$. 

After fixing conventions for the measures, by the defining formula for product measure in Subsection \ref{subsec:5.2}, the volume of the cylinder-shaped fundamental domain is computed in two steps: measure the length of the intersection of the fundamental domain and each fiber, obtaining a function on the base space which turns out to be a constant multiple of the characteristic function of $D$; next, integrate this function over the base space. The resulting volume is as stated in the theorem. 
\end{proof}

%
%For Siegel--Seifert subgroup $\Gamma \subset  \CG\times_{Z(\CG)}\CH\CR/\CH$, the volume of $\Gamma\backslash \CX$ is given by 
%\[
%\vol(\Gamma\backslash \CX)=\pm\vol((\Gamma\cap\ker\eta)\backslash  (\CH\CR/\CH)    )\cdot \chi(\eta(\Gamma)). 
%\]
%
%In particular, for Siegel--Seifert subgroup $\Gamma$ arising from (Definition \ref{def:arisesfrom}) $\Psptnr$, the volume of $\Gamma\backslash \CX$ is given by 
%\[
%\vol(\Gamma\backslash \CX)=\pm \chi(\eta(\Gamma)). 
%\]
%
%
%
%Here  $\eta:  \CG\times_{Z(\CG)}\CH\CR/\CH \to \CG/Z(\CG)\cong \Psptnr$ is given in Proposition \ref{prop:exactseq}. 

\begin{cor}
For Siegel--Seifert subgroups $\Gamma_i$ arising from $\Psptnr$, the volumes of $\Gamma_i\backslash \CX$ are rationally related. 
\end{cor}

\begin{remark}
Related to $\mathrm{Eu}$, there is a unique invariant measure on some Lie groups $G$ (not the corresponding homogeneous spaces! ) called the Euler--\Poincare measure \cite{serrearithgroup}, which  computes $\chi(\Gamma)$ for some $\Gamma$. In \cite{harder}, Harder found an explicit formula for $\chi(\Gamma)$ in terms of zeta function. In particular, 
\[
\chi(\mathrm{Sp}(2n,\mathbb{Z}))=\prod_{k=1}^n\zeta(1-2k), 
\]
\[
\chi(\mathrm{Sp}(4,\mathbb{Z}))=-\frac{1}{1440}. 
\]
This is why the  sign is added in my convention. 
For the definition of Euler characteristic (with values in $\BQ$) of a group with torsion, see   \cite[Chapter~IX]{gtm87}. 
\end{remark}

\section{Examples}
\label{sec:6}

In this section, we provide examples of Siegel--Seifert subgroups. It is shown in Proposition \ref{prop:siegelseifertarisingfrom}  that 
for any cocompact torsion-free lattice $L$ in $\Psptnr$, 
the group $\Gamma=L^\uparrow$ arising from $L$  is Siegel--Seifert. 

Proposition \ref{prop:siegelseifertarisingfrom} is discussed in Subsection \ref{subsec:6.1}. 
In 1963, Borel \cite{borel} showed that a simply connected Riemannian symmetric space $M$ always has a compact
Clifford--Klein form. Classification of lattices for some simple real Lie groups is summarized by  Audibert in \cite{splattices} (see also his thesis),  based on several general results in \cite{morris}. This is briefly sketched in Subsection \ref{subsec:6.2}. 
Subsection \ref{subsec:6.2} is aimed at providing examples of  Siegel--Seifert subgroups, starting from cocompact lattices in $\Sptnr$. The process goes as follows: Take a cocompact lattice of $\Sptnr$ in the  classification, take its image in $\Psptnr$ which is again a cocompact lattice in $\Psptnr$, 
 replace it with a torsion-free subgroup of finite index which is again a cocompact lattice, and then carry out the ``arise from'' construction.

\subsection{Subgroups arising from cocompact lattices of $\Psptnr$}
\label{subsec:6.1}

The construction in Definition \ref{def:arisesfrom} provides examples of Siegel--Seifert subgroups. 
Following \cite{morris}, a subgroup $\Gamma$ of a (linear, semisimple, having only finitely many
connected components) Lie group $G$ is a \textbf{lattice} if $\Gamma$ is discrete and $\Gamma\backslash G$ has finite volume with respect to the Haar measure on $G$. A lattice $\Gamma \subset G $ is \textbf{uniform} or \textbf{cocompact} if $\Gamma\backslash G$ is compact, and \textbf{non-uniform} otherwise. 

\begin{lem}
Given any lattice $L$ in $\Psptnr$, 
the group $\Gamma=L^\uparrow$ arising from $L$ satisfies the followings: 

(i) $\Gamma \subset  \CG\times_{Z(\CG)}\CH\CR/\CH$ is discrete, and the action $\Gamma\curvearrowright 
\CX$ is properly discontinuous, 

(ii) the action $\eta(\Gamma)\curvearrowright \fH_n$ is properly discontinuous, and

(iii) $\Gamma\cap\ker\eta$ is infinite cyclic. 
\label{lem:latticepsptnr}
\end{lem}

\begin{proof}
Part (iii) is immediate from definition. 

For part (ii), $\eta(\Gamma)=L$ is discrete. We need the following general fact: Suppose $G$ is a locally compact topological group, $H\subset G$ is a compact subgroup and $F\subset G$ is a discrete subgroup, then the action of $F$ on $G/H$ by left multiplication is properly discontinuous. Hence the action $\eta(\Gamma)\curvearrowright \fH_n$ is properly discontinuous. 

For part (i), suppose $\Gamma \subset  \CG\times_{Z(\CG)}\CH\CR/\CH$ is not discrete. Take a  sequence of elements of $\Gamma$ converging to the identity. Applying $\eta$, we obtain a sequence in $\eta(\Gamma)$ converging to the identity, which is eventually the identity as $\eta(\Gamma)$ is discrete. It follows that the sequence in $\Gamma$ is eventually contained in $\Gamma\cap\ker\eta$, which is infinite cyclic. As the sequence in $\Gamma$ converges to the identity, it must be the identity eventually. 

Thus $\Gamma$ is discrete. The same argument as in part (ii) shows that the action of $\Gamma$ is properly discontinuous. 
\end{proof}

For any lattice $L$ in $\Psptnr$, by Selberg's Lemma (see \cite[Section~4.8]{morris}), $L$ has a torsion-free subgroup of finite index.  Every finite-index subgroup of a lattice is a lattice.

Starting with any cocompact lattice in $\Psptnr$, we may replace it with a torsion-free subgroup of finite index which is again a cocompact lattice.

\begin{prop}
Given any cocompact torsion-free lattice $L$ in $\Psptnr$, 
the group $\Gamma=L^\uparrow$ arising from $L$  is Siegel--Seifert. 
\label{prop:siegelseifertarisingfrom}
\end{prop}

\begin{proof}
By Lemma \ref{lem:latticepsptnr} above and the fact that  $\eta(\Gamma)=L$, it remains to show that 
 $\Gamma \subset  \CG\times_{Z(\CG)}\CH\CR/\CH$ is  torsion-free, and the action $\Gamma\curvearrowright 
\CX$ is cocompact.

Any torsion element $x$ in $\Gamma$ is mapped to a torsion element in $\eta(\Gamma)$, which must be the identity since $\eta(\Gamma)$ is torsion free. Hence $x$ is indeed in the kernel $\Gamma\cap\ker\eta$, which is infinite cyclic. The only possibility is that $x$ itself is the identity. 

To see that the action $\Gamma\curvearrowright 
\CX$ is cocompact, note that the construction of a fundamental domain in Proposition 
\ref{prop:funddom} still works, producing a compact fundamental domain. 
\end{proof}

\subsection{Classification of lattices in $\Sptnr$}
\label{subsec:6.2}

This subsection is aimed at providing examples of  Siegel--Seifert subgroups, starting from cocompact lattices in $\Sptnr$.

The image in $\Psptnr$ of a lattice in $\Sptnr$ is a lattice in $\Psptnr$; the inverse image in $\Sptnr$ of a lattice in $\Psptnr$ is a lattice in $\Sptnr$. From considerations in the previous subsection we are lead to consider the classification of (uniform) lattices in $\Sptnr$, up to (wide) commensurability. 
This is summarized  by  Audibert in \cite[Proposition~2.1.1]{splattices} (see also his thesis), largely based on a general  description of arithmetic subgroups of almost every classical Lie group in \cite[Chapter~18]{morris}. 
%Here we only include some crucial definitions needed in the statement of the results. 
For more details and definitions, see \cite{morris}, 
\cite{splattices}, \cite{spuniformlattices} and the references therein.

%A number field $F$ is said to be \textbf{totally real} if all its embeddings in $\BC$ are real. 

%\begin{prop}[\cite{splattices}]
%Let $n \geq 2$. Every lattice in $\Sptnr$ is widely commensurable with $\mathrm{SU}(I_n,\overline{\phantom{a}};  \CO)$ for
%$\CO$ an order in a quaternion algebra $A$ over a totally real number field $F$ such that $A$ splits
%over exactly one real place of $F$. 
%\end{prop}

\begin{prop}[\cite{spuniformlattices}]
Let $n \geq 2$. The uniform lattices of $\Sptnr$ are widely
commensurable with $\mathrm{SU}(I_n,\overline{\phantom{a}};  \CO)$ for $\CO$ an order of a quaternion algebra $A$
over a totally real number field $F\neq \BQ$ such that $A$ splits at exactly one real
place of $F$. 
\end{prop}

Here is a glimpse of the proof. 
 Margulis’ Arithmeticity Theorem  implies that all lattices of $\Sptnr$ are
arithmetic. Thus by classification results in \cite{morris}, 
 lattices of $\Sptnr$ are
widely commensurable with certain special unitary groups. 
Classification of non-degenerate $\overline{\phantom{a}}$-Hermitian forms on quaternion algebra $A$ identifies these special unitary groups with the ones in the propositions above.

\begin{remark}
In  \cite[Section~18.7]{morris}, another concrete construction  of  a cocompact, arithmetic lattice is given. This construction, or alternatively Proposition 18.7.2 together with the accidental isomorphism 
\[
\mathrm{SO}(3,2)\cong \pspfr, 
\]
can be applied to obtain a candidate for a concrete computable example of a Siegel--Seifert subgroup $\Gamma$ arising from $\pspfr$. 
\end{remark}

\begin{remark}
As the anonymous referee points out, any general Siegel--Seifert subgroup $\Gamma$ (not necessarily arising from subgroups of $\Psptnr$) is an extension of $\eta(\Gamma)$ by an infinite cyclic group. Such an extension is classified by $H^2(\eta(\Gamma); \mathbb{Z})$. 
\end{remark}

\section{Volume of representations}
\label{sec:7}

Take    any connected real Lie group $G$, acting on the contractible 
homogeneous space  $G/H$ where $H$ is a fixed maximal compact
subgroup of $G$,  with a  $G$-invariant volume form $\omega$ on $G/H$.  
Associated to this data, a volume $\vol_G(M,\rho)=\vol_{(G,G/H,\omega)}(M,\rho)\in \BR$ can be assigned to any representation $\rho:\pi_1(M)\to G$ where $M$ is any oriented
closed smooth manifold of the same dimension as $G/H$. 
As $\rho$ ranges,  the supremum of absolute value of the volume function is  the $G$-representation volume $\mathrm{V}(M,G)\in[0,+\infty]$ of the manifold $M$. 
Representation volumes  have been   used to study problems related to  mapping degree.

In this section, we exhibit structure of representation spaces and also  results on the volume of representations associated to $\gxzghrh$, as corollaries to \cite{dlsw}. This is done by showing that our extended group $\gxzghrh$ is isomorphic to theirs, for certain choices of $G$ in their paper.  The main results of this section are Corollaries \ref{cor:oddeveneuler} and \ref{cor:oddevenrigid}. 

We begin with a brief introduction to volume of representations, and then the main results of this section. The two cases,  where $n$ is odd or even,  are dealt with separately in Subsections \ref{subsec:7.3} and \ref{subsec:7.4}.

As a consequence of considerations on representation volume, we show in Proposition \ref{prop:seifertpropertyd} that any Siegel--Seifert manifold $\Gamma\backslash \mathcal{X}$ has Property D (Definition \ref{defn:propertyd}), in Subsection \ref{subsec:7.5}.

\subsection{Introduction to  volume of representations}

Let $G$ be  any connected real Lie group, acting on the contractible 
homogeneous space  $G/H$ where $H$ is a fixed maximal compact
subgroup of $G$,  with a  $G$-invariant volume form $\omega$ on $G/H$ chosen and fixed. For any oriented
closed smooth manifold $M$ of applicable dimension, and any representation $\rho:\pi_1(M)\to G$, let
\[
\vol_G(M,\rho)=\vol_{(G,G/H,\omega)}(M,\rho)\in \BR
\]
be the \textbf{volume of representation} $\rho$, defined as follows:
Take a $\rho$-equivariant map from the universal cover $\tilde{M}$ of $M$ to $G/H$, pull back the invariant volume form to $\tilde{M}$, and integrate over a fundamental domain. 

As $\rho$ ranges over the \textbf{representation space} $\mathrm{Rep}(\pi_1(M), G)$ of all homomorphisms from $\pi_1(M)$ to $G$ (with the algebraic-convergence
topology),  the supremum of absolute value of the volume function is defined to be the $G$-\textbf{representation volume} $\mathrm{V}(M,G)\in[0,+\infty]$ of the manifold $M$. 
See \cite{dlsw} for detailed definitions and properties of volume of representations.

%\vol_G=\vol_{(G,G/H,\omega)}:\mathrm{Rep}(\pi_1(M),G) \to \BR

For $M$   an oriented connected closed $3$-manifold,  representation $\rho$ from $\pi_1(M)$ to $\widetilde{\mathrm{SL}(2,\mathbb{R})}\times_\BZ \BR$, and some fixed invariant volume form on $\widetilde{\mathrm{SL}(2,\mathbb{R})}$, the volume of $\rho$ and the corresponding 
 \textbf{Seifert volume} of $M$ are originally introduced by Brooks and Goldman \cite{goldman, brooksgoldmangv, brooksgoldman}. The Seifert volume behaves somehow similarly to Gromov's simplicial volume \cite{gromov}; in particular, we have a domination inequality \cite[Corollary~3.2]{dlsw} of representation volumes, which is of the form 
$
\mathrm{V}(M',G)\geq |\mathrm{deg}(f)| \cdot \mathrm{V}(M,G)
$
for $f:M'\to M$. As a consequence, Seifert volume, and more generally representation volumes, are also used to study problems related to (the set of) mapping degrees (see \cite{virposseivol, dlsw} and the references therein).

The original motivation of this paper is to generalize the works on the volume of $\widetilde{\mathrm{SL}(2,\mathbb{R})}\times_\BZ \BR$-representations to the symplectic case. 
Our volume formula \ref{thm:vol} is also partially motivated by  \cite[Theorem~3.2]{volofseirep}, which, very roughly speaking, states that the volume is the product of ``the length of fiber'' and ``the volume or Euler class of the base''. Conjecturally, a similar formula may hold for {$\gxzghrh$}-representations of Siegel--Seifert manifolds, with the ``Euler class'' of the induced $\CG/\CZ(\CG)$-representations of the base manifold suitably interpreted.

\subsection{Overview of the results}

%HERE

For any finitely generated group $\pi$, if $n$ is odd, we consider maps between the representation spaces
\[
\mathrm{Rep}(\pi,\CG)\to\mathrm{Rep}(\pi,\CG\times_{Z(\CG)}\CH\CR/\CH )\to\mathrm{Rep}(\pi,\CG/Z(\CG)), 
\]
which are naturally induced by the group homomorphisms
$\CG\to \CG\times_{Z(\CG)}\CH\CR/\CH \to\CG/Z(\CG)$. 

If $n$ is even,  there are maps between the representation spaces 
\[
\mathrm{Rep}(\pi,\Sptnrt/\{\pm\id\}) \to \mathrm{Rep}(\pi,\CG\times_{Z(\CG)}\CH\CR/\CH) \to \mathrm{Rep}(\pi,\Psptnr). 
\]

%For any finitely generated group $\pi$, there are maps between the representation spaces 
%\[
%\mathrm{Rep}(\pi,\Sptnrt/\{\pm\id\}) \to \mathrm{Rep}(\pi,\CG\times_{Z(\CG)}\CH\CR/\CH) \to \mathrm{Rep}(\pi,\Psptnr). 
%\]
%

The main results of this section are the followings.

\begin{cor}[Proposition 5.1 in \cite{dlsw} for $G=\Sptnrt$ if $n$ is odd, and  for $G=\Sptnrt/\{\pm\id\}$ if $n$ is even]
There is a characteristic class for $\Psptnr$-representations of finitely generated groups, namely, a natural assignment
$$e: \mathrm{Rep}(\pi,\Psptnr) \rightarrow H^2(\pi;\BZ)$$
for any finitely generated group $\pi$. Moreover, the following statements are true:

(1) The space of representations $\mathrm{Rep}(\pi,\Psptnr)$ is a finite union of path-connected components of an affine real algebraic variety. The characteristic class $e$ is constant over each path-connected component of $\mathrm{Rep}(\pi,\Psptnr)$.

(2) For odd $n$, the space of representations $\mathrm{Rep}(\pi,\Sptnrt)$ is an $H^1(\pi;\BZ)$-principal bundle over the union of the path-connected components of $\mathrm{Rep}(\pi,\Psptnr)$ on which $e$ is trivial.  

For even $n$, the space of representations $\mathrm{Rep}(\pi,\Sptnrt/\{\pm\id\})$ is an $H^1(\pi;\BZ)$-principal bundle over the union of the path-connected components of $\mathrm{Rep}(\pi,\Psptnr)$ on which $e$ is trivial.

(3) The space of representations $\mathrm{Rep}(\pi,\CG\times_{Z(\CG)}\CH\CR/\CH)$ is an $H^1(\pi;\BR)$-principal bundle over the union of the path-connected components of $\mathrm{Rep}(\pi,\Psptnr)$ on which $e$ is torsion.
\label{cor:oddeveneuler}
\end{cor}

\begin{cor}[Theorem 7.1 in \cite{dlsw} for $G=\Sptnrt$ if $n$ is odd, and  for $G=\Sptnrt/\{\pm\id\}$ if $n$ is even]
 For any closed oriented smooth manifold $M$ (of applicable dimension), the volume of representations
$$\operatorname{vol}_{\CG\times_{Z(\CG)}\CH\CR/\CH}:\mathrm{Rep}(\pi_1(M),\CG\times_{Z(\CG)}\CH\CR/\CH)\to\mathbb{R}$$
is constant on every path-connected component of the representation space $\mathrm{Rep}(\pi_1(M),\CG\times_{Z(\CG)}\CH\CR/\CH)$.
\label{cor:oddevenrigid}
\end{cor}

%NOTE:label

%
%
%\begin{cor}[Theorem 7.1 in \cite{dlsw} for $G=\Sptnrt$ if $n$ is odd]
%Let $n$ be odd. 
% For any closed oriented smooth manifold $M$ (of applicable dimension), the volume of representations
%$$\operatorname{vol}_{\CG\times_{Z(\CG)}\CH\CR/\CH}:\mathrm{Rep}(\pi_1(M),\CG\times_{Z(\CG)}\CH\CR/\CH)\to\mathbb{R}$$
%is constant on every path-connected component of the representation space $\mathrm{Rep}(\pi_1(M),\CG\times_{Z(\CG)}\CH\CR/\CH)$.
%\label{cor:oddrigid}
%\end{cor}
%
%
%\begin{cor}[Theorem 7.1 in \cite{dlsw} for $G=\Sptnrt/\{\pm\id\}$ if $n$ is even]
%Let $n$ be even. 
% For any closed oriented smooth manifold $M$ (of applicable dimension), the volume of representations
%$$\operatorname{vol}_{\CG\times_{Z(\CG)}\CH\CR/\CH}:\mathrm{Rep}(\pi_1(M),\CG\times_{Z(\CG)}\CH\CR/\CH)\to\mathbb{R}$$
%is constant on every path-connected component of the representation space $\mathrm{Rep}(\pi_1(M),\CG\times_{Z(\CG)}\CH\CR/\CH)$.
%\label{cor:evenrigid}
%\end{cor}
%

Corollaries \ref{cor:oddeveneuler} and \ref{cor:oddevenrigid} are corollaries to Proposition 5.1 and Theorem 7.1 in \cite{dlsw}. 
For any finitely generated group $\pi$,  
Proposition 5.1 points out the relations between representation spaces
$\mathrm{Rep}(\pi,G) $, $ \mathrm{Rep}(\pi,G_{\mathbb{R}}) $ and $ \mathrm{Rep}(\pi,\overline{G})$, in terms of a characteristic class $e_Z$. 
This characteristic class (constructed in the proof of  Proposition 5.1) can be considered a generalization of the Euler class of a representation to $\mathrm{PSL}(2,\BR)$ or more generally a representation to $\mathrm{Homeo}_+(S^1)$ \cite{ghys}. 
Theorem 7.1 shows that the volume of representations
$\operatorname{vol}_{G_{\mathbb{R}}}:\mathrm{Rep}(\pi_1(M),G_{\mathbb{R}})\to\mathbb{R}$
is constant on every path-connected component of the representation space $\mathrm{Rep}(\pi_1(M),G_{\mathbb{R}})$.

It is important to note that the definition of a semisimple Lie group in \cite{dlsw} is different from the definition given by  some other authors (for example, \cite{morris}). This definition allows one to consider some Lie groups with infinite center, or some Lie groups that do not embed in a matrix group, to be semisimple. In contrast, some authors require semisimple Lie groups to have finite center.

\subsection{Case I: $n$ is odd}
\label{subsec:7.3}

In this subsection, we treat the case where $n$ is odd. 
We prove that after fixing identifications $\CH\CR/\CH\to \BR$ and $Z_{\BR}:=Z(\CG)\otimes_{\BZ}\BR\to \BR$ (using the notation of \cite{dlsw}, to be explained in the followings), our extended group $\gxzghrh$ and their extended group $\CG_\BR:=\CG\times_{Z(\CG)} Z_\BR$ are both identified with the same quotient of $\CG\times \BR$. As a result, Proposition 5.1 and Theorem 7.1 in \cite{dlsw} in this case directly apply; however, the other case needs some modifications in order for these results to apply.

Recall from Lemma \ref{lem:center} that the center of $\Sptnrt$ is $\pi^{-1}(\{\pm\mathrm{id}\})$, and for $n$ odd $\pi^{-1}(\{\pm\mathrm{id}\})\cong \BZ$. The point is that in this case $-\mathrm{id}\notin \Sun$. In the language of the proof of Lemma-Definition \ref{lemdef:ghr}, this center is generated by $\alpha:=$ $\frac{1}{2}$-th lift of $-\mathrm{id}$. We introduce the lifted determinant map to make the definition of $\alpha$ neat: 

\begin{lem}
Let $\widetilde{\mathrm{det}}:\Unt\to \BR$ be the (topological) lift of $\det:\Un\to S^1$ such that the identity is sent to the identity. Then $\dett$ is indeed a Lie group homomorphism, and $\ker(\dett)=\Sun$. 
\end{lem}

With $\dett$ in hand, we may now define: 
\[
\alpha \in Z(\CG) \text{ is the lift of } -\mathrm{id} \text{ with } \dett(\alpha)=\frac{1}{2}. 
\]

In accordance with \cite{dlsw}, we adopt the following 
\begin{conv}
In this subsection, group operations in $Z=Z(\CG)$, $Z_{\BR}=Z(\CG)\otimes_{\BZ}\BR$, $\hrh$ and $\BR$ will be written additively. 
\end{conv}

The  extended group $\CG_\BR$ constructed in \cite{dlsw}  is defined to be 
\[
\CG_\BR:=\CG\times_{Z(\CG)} Z_\BR=
\frac{\CG\times Z_\BR}{  \{(z, -z\otimes 1) \mid z\in Z(\CG)\}   }, 
\]
where $Z_{\BR}=Z(\CG)\otimes_{\BZ}\BR$. 
After fixing an auxiliary identification
\[
Z_{\BR} \to \BR, \quad \alpha \otimes b \mapsto \frac{1}{2}b, 
\]
the extended group $\CG_\BR$ becomes 
\[
\CG_\BR \cong
\frac{\CG\times\BR}{      \{(k\alpha, -\frac{k}{2})\mid k\in \BZ\}       }, 
\]
as this center is generated by $\alpha$.

Written additively, our extended group is 
\[
\CG\times_{Z(\CG)}\CH\CR/\CH:=  \frac{\CG\times \CH\CR/\CH}{(i,-(\iota))Z(\CG)}, 
\]
where $\iota:Z(\CG)\to \CH\CR \to \hrh$. 
By the lemma above, $\dett:\CH\CR\to \BR$ factorizes as $\dett:\CH\CR\to \hrh \to \BR$. We take the second homomorphism $\hrh \to \BR$ to be the identification. Under this identification,  $\iota:Z(\CG)\to \CH\CR \to \hrh\to \BR$ sends $\alpha $ to $\frac{1}{2}$. Thus
\[
\CG\times_{Z(\CG)}\CH\CR/\CH \cong 
\frac{\CG\times \BR}{      \{(k\alpha, -\frac{k}{2})\mid k\in \BZ\}       }
\cong \CG_\BR. 
\]

%result cor

Consequently, 
 Proposition 5.1 and Theorem 7.1 in \cite{dlsw} in this case directly apply.

\subsection{Case II: $n$ is even}

\label{subsec:7.4}

In Lemma \ref{lem:center}, it is proved that the center of $\Sptnrt$ is $\pi^{-1}(\{\pm\mathrm{id}\})$ and for $n $ even, $ \pi^{-1}(\{\pm\mathrm{id}\})\cong \BZ\oplus \BZ/2\BZ$. 
The center of $\Sptnrt/\{\pm\id\}$ is equal to $Z(\CG)/\{\pm\id\}$, which is isomorphic to $\BZ$. It is generated by the element $\beta:=$ image in $Z(\CG)/\{\pm\id\}$ of the first lift of $\id$, which is also the image in $Z(\CG)/\{\pm\id\}$ of the first lift of $-\id$. In terms of $\dett$, we may define: 
\[
\beta\in Z(\CG)/\{\pm\id\} \text{ is the element represented by the lift of } \id \text{ with } \dett =1. 
\]

We adopt the following 
\begin{conv}
In this subsection, group operations in $Z(\CG)/\{\pm\id\}$, $(Z(\CG)/\{\pm\id\})\otimes_\BZ \BR$, $Z(\CG)$, $\hrh$ and $\BR$ will be written additively. 
\end{conv}

For $G=\Sptnrt/\{\pm\id\}=\CG/\{\pm\id\}$, the construction in \cite{dlsw} gives
\[
G_\BR:=
\frac{\left(\CG/\{\pm\id\}\right)\times\left((Z(\CG)/\{\pm\id\})\otimes_\BZ \BR\right)}{ \left\{(z,-z\otimes 1)\mid z\in Z(\CG)/\{\pm\id\}\right\}  }. 
\]
Fix the identification
\[
(Z(\CG)/\{\pm\id\})\otimes_\BZ \BR\to \BR, \quad \beta\otimes b\to b. 
\]
Then $G_\BR$ is identified with
\[
G_\BR\cong
\frac{\left(\CG/\{\pm\id\}\right)\times\BR}{ \left\{(k\beta, -k)\mid k\in \BZ\right\}  }. 
\]

On the other hand, our construction in this case gives
\[
\begin{aligned}
\CG\times_{Z(\CG)}\CH\CR/\CH := & \frac{\CG\times \CH\CR/\CH}{(i,-(\iota))Z(\CG)}\\
\cong & \frac{(\CG\times \CH\CR/\CH)/\{(\pm \id, 0)\}}{((i,-(\iota))Z(\CG))/\{(\pm \id, 0)\}}\\
\cong & \frac{\left(\CG/\{\pm\id\}\right)\times \CH\CR/\CH}{((i,-(\iota))Z(\CG))/\{(\pm \id, 0)\}}. \\
\end{aligned}
\]
As in the first case, we can pick the identification $\hrh\to \BR$ that is induced from $\dett$. Hence
\[
\begin{aligned}
\CG\times_{Z(\CG)}\CH\CR/\CH\cong & \frac{\left(\CG/\{\pm\id\}\right)\times \BR}{\left\{(z,-\dett(z))\mid z\in Z(\CG)\right\}/\{(\pm \id, 0)\}}\\
\cong & \frac{\left(\CG/\{\pm\id\}\right)\times \BR}{\{(k\beta, -k)\mid k\in \BZ\}}\\
\cong & G_\BR.  \\
\end{aligned}
\]

Accordingly, the  corollaries are obtained.

\subsection{Applications to mapping degree problem}
\label{subsec:7.5}

We will show that Siegel--Seifert manifolds satisfy the following property: 
\begin{defn}
A closed orientable smooth manifold $N$ has \textbf{Property D} if for all closed orientable smooth manifold $M$ of the same dimension, the set of mapping degrees
\[
\mathcal{D}(M,N)=\{|\mathrm{deg}(f)|\mid f\in C^\infty(M,N)\}
\]
is always finite. 
\label{defn:propertyd}
\end{defn}

Among $2$-dimensional manifolds, those with Property D are those of the geometry $\mathbb{H}^2$. 
In dimension $3$, manifolds with Property D are those which contain non-geometric prime factors, or prime factors of the geometry $\mathbb{H}^3$ or $\sltrt$, by \cite[Corollary~1.6]{virposseivol}. In \cite[Theorem~1.3]{dlsw}, it is proved that if $(X, G)$ is a geometry where $X$ is contractible and $G$ is semisimple, then
every orientable closed manifold locally modeled on $(X, G)$ has Property D. See \cite{dlsw} for an overview of results on Property D.

We need the following lemma: 
\begin{lem}[\cite{dlsw}, Proposition 3.1(1)]
Let $M$, $M'$ be closed oriented smooth manifolds of the same dimension as $\CX$. 
For any smooth
map $f : M' \to M$  and any representation $\rho:\pi_1(M)\to \CG\times_{Z(\CG)}\CH\CR/\CH$, 
\[
\mathrm{vol}_{\CG\times_{Z(\CG)}\CH\CR/\CH}(M', f^*(\rho)) = \mathrm{deg}(f) \cdot \mathrm{vol}_{\CG\times_{Z(\CG)}\CH\CR/\CH}(M, \rho),
\]
where $\mathrm{deg}(f)$ is the signed mapping degree of $f$. 
As a consequence, we have the domination inequality
\[
\mathrm{V}(M',\CG\times_{Z(\CG)}\CH\CR/\CH)\geq|\mathrm{deg}(f)|\cdot\mathrm{V}(M,\CG\times_{Z(\CG)}\CH\CR/\CH).
\]
\end{lem}

The key point in its proof is that for any developing map $D_\rho: \widetilde{M} \to \CX$, the composed map $D_\rho \circ \tilde{f}: \widetilde{M'} \to \widetilde{M} \to \CX$ is a developing map for $f^*(\rho)$.

\begin{prop}
Any Siegel--Seifert manifold $\Gamma\backslash \mathcal{X}$ has Property D. 
\label{prop:seifertpropertyd}
\end{prop}

\begin{proof}
Take a Siegel--Seifert manifold $\Gamma\backslash \mathcal{X}$ with corresponding Siegel--Seifert subgroup $\Gamma \subset  \CG\times_{Z(\CG)}\CH\CR/\CH$. It follows directly from the definition that  $\Gamma\backslash \mathcal{X}$ is a closed oriented smooth manifold of applicable dimension, having the same dimension as $\CX$. 

For the embedding $\rho:\Gamma \to  \CG\times_{Z(\CG)}\CH\CR/\CH$ considered as a representation, 
$ \operatorname{vol}_{\CG\times_{Z(\CG)}\CH\CR/\CH}(\Gamma\backslash \mathcal{X}, \rho)$ is just the volume of $\Gamma\backslash \mathcal{X}$ in the sense of previous sections, which is nonzero. Hence after taking supremum, the resulting $\CG\times_{Z(\CG)}\CH\CR/\CH$-representation volume $\mathrm{V}(\Gamma\backslash \mathcal{X},\CG\times_{Z(\CG)}\CH\CR/\CH)$ is nonzero. 
Plugging in the domination inequality above, we see that  $\Gamma\backslash \mathcal{X}$ has Property D. 
\end{proof}

\bibliographystyle{amsalpha}
\bibliography{spbib}
\end{document}